\documentclass[11pt,a4paper]{amsart}		

\usepackage[applemac]{inputenc}					 
\usepackage{lmodern}						 
\usepackage[english]{babel}					

\usepackage[weather]{ifsym}

\usepackage{hyperref,enumitem,soul}

\hypersetup{
	colorlinks,
	linkcolor={red!50!black},
	citecolor={blue!50!black},
	urlcolor={blue!80!black}
}

\usepackage{tikz}
\usetikzlibrary{calc, svg.path, patterns}
\usepackage{todonotes}

\usepackage{amsmath,amsfonts,amssymb,amsthm}

\usepackage[left=3cm,right=3cm,top=3cm,bottom=3cm]{geometry}

 \usepackage{mathtools}

\theoremstyle{plain}
\newtheorem{theorem}{Theorem}
\newtheorem{proposition}{Proposition}[section]

\newtheorem{lemma}[proposition]{Lemma}

\newtheorem{definition}[proposition]{Definition}
\newtheorem{remark}[proposition]{Remark}
\newtheorem{example}{Example}

\numberwithin{equation}{section}

\newcommand{\R}{\mathbb{R}}					
\newcommand{\C}{\mathbb{C}}					
\newcommand{\N}{\mathbb{N}}

\def\norm#1#2{\|#1\|_{#2}}

\def\refer#1{~\ref{#1}}
\def\refeq#1{~(\ref{#1})}
\def\ccite#1{~\cite{#1}}
\def\pagerefer#1{page~\pageref{#1}}

\def\suite#1#2#3{(#1_{#2})_{#2\in {#3}}}

\def\inte#1{
\displaystyle\mathop{#1\kern0pt}^\circ }

\def\sumetage#1#2{\sum_{\substack{{#1}\\{#2}}}}

\let\d=\delta

\let\lam=\lambda

\let\D=\Delta

\let\Lam=\Lambda
\let\S=\Sigma

\let\wt=\widetilde
\let\wh=\widehat

\def\cA{{\mathcal A}}

\def\cC{{\mathcal C}}
\def\cD{{\mathcal D}}

\def\cF{{\mathcal F}}

\def\DD{{\mathbb D}}

\def\cL{{\mathcal L}}

\def\cO{{\mathcal O}}

\def\cS{{\mathcal S}}

\def\cW{{\mathcal W}}
\def\cX{{\mathcal X}}

\def\S{{\mathop{\mathbb  S\kern 0pt}\nolimits}}

\def\virgp{\raise 2pt\hbox{,}}
\def\cdotpv{\raise 2pt\hbox{;}}
\def\eqdefa{\buildrel\hbox{\footnotesize def}\over =}

\def\C{\mathop{\mathbb C\kern 0pt}\nolimits}
\def\EE{\mathop{{\mathbb E \kern 0pt}}\nolimits}
\def\K{\mathop{\mathbb K\kern 0pt}\nolimits}
\def\N{\mathop{\mathbb N\kern 0pt}\nolimits}
\def\Q{\mathop{\mathbb Q\kern 0pt}\nolimits}
\def\R{{\mathop{\mathbb R\kern 0pt}\nolimits}}
\def\SS{\mathop{\mathbb S\kern 0pt}\nolimits}
\def\ZZ{\mathop{\mathbb Z\kern 0pt}\nolimits}
\def\TT{\mathop{\mathbb T\kern 0pt}\nolimits}
\def\P{\mathop{\mathbb P\kern 0pt}\nolimits}
\def \H{{\mathop {\mathbb H\kern 0pt}\nolimits}}
\newcommand{\ds}{\displaystyle}

\newcommand{\Z}{{\ZZ}}

\newcommand{\beq}{\begin{equation}}
\newcommand{\eeq}{\end{equation}}
\newcommand{\ben}{\begin{eqnarray}}
\newcommand{\een}{\end{eqnarray}}
\newcommand{\beno}{\begin{eqnarray*}}
\newcommand{\eeno}{\end{eqnarray*}}
\newcommand{\bqs}{\begin{equation*}}
\newcommand{\eqs}{\end{equation*}}
\newcommand{\andf}{\quad\hbox{and}\quad}
\newcommand{\with}{\quad\hbox{with}\quad}

\def \cFH {\cF_\H}

\def\equivH#1 {\buildrel\hbox{\tiny {$#1$}}\over \equiv}
\def\simH#1 {\buildrel\hbox{\footnotesize {$#1$}}\over \sim}

\title[Restriction theorems and Strichartz estimates on the Heisenberg group]
{Strichartz estimates and    Fourier Restriction Theorems on the Heisenberg group}

\author[H. Bahouri]{Hajer Bahouri}
\address[H. Bahouri]
{CNRS  \&  Sorbonne Universit\'e  \\
 Laboratoire Jacques-Louis Lions (LJLL) UMR  7598 \\
4, Place Jussieu\\
75005 Paris, France.}
\email{hajer.bahouri@ljll.math.upmc.fr}
\author[D. Barilari]{Davide Barilari}\address[D. Barilari]%
{Dipartimento di Matmatica "Tullio Levi-Civita" \\
      Universit{\`a} di Padova \\
 Via Trieste 63 \\
 Padova, Italy}
\email{barilari@math.unipd.it}
\author[I. Gallagher]{Isabelle Gallagher}
\address[I. Gallagher]%
{DMA UMR 8553, \'Ecole normale sup\'erieure, CNRS, PSL Research University, 75005 Paris 
 \\
and UFR de math\'ematiques, Universit\'e Paris-Diderot, Sorbonne Paris-Cit\'e, 75013 Paris, France.}
\email{gallagher@math.ens.fr}

\begin{document}

\setstcolor{red}

\begin{abstract} 
This paper is dedicated to the proof of Strichartz  estimates on the Heisenberg group~$\H^d$  for the linear Schr\"odinger  and wave equations involving the sublaplacian.  The Schr\"odinger equation on $\H^d$ is an example of a  totally non-dispersive evolution equation: for this reason the classical approach that permits to obtain Strichartz estimates from dispersive estimates is not available.
Our approach, inspired  by the Fourier transform restriction method initiated in\ccite{Tomas}, is based on  Fourier restriction theorems on~$\H^d$, using  the non-commutative Fourier transform on the Heisenberg group. 
 It enables us to obtain also an anisotropic  Strichartz estimate  for the wave equation,  for a larger range of indices than was previously known. 
\end{abstract}
\maketitle
 
\section{Introduction}\label {intro}
\setcounter{equation}{0}
\subsection{Strichartz  estimates}\label {introst} In the past decades,  Strichartz  estimates  for linear evolution  equations such as the Schr\"odinger and wave equations, have been   a central tool    in the study  of  semilinear and quasilinear equations, which appear in numerous physical applications.  In many cases and particularly in $\R^n$, the proof of  those inequalities, which involve space-time Lebesgue norms, is a combination of an abstract functional analysis argument known as the~$TT^*$-argument (see\ccite{ginibrevelo}) and of a dispersive estimate.  Concerning the Schr\"odinger equation on~$\R^{n}$  \label{schroeqRn}
$$ (S) \qquad \left\{
\begin{array}{c}
i\partial_t u -\D u = 0\\
u_{|t=0} = u_0\,,
\end{array}
\right.  
$$
the dispersive estimate writes  (for $t \neq 0$)
\beq
\label {dispSR0}
\|u(t,\cdot)\|_{L^\infty(\R^{n})} \leq \frac 1 {(4\pi|t|)^{\frac n 2}} \|u_0\|_{L^1(\R^{n})}\,,
\eeq
and can be easily derived from the explicit expression of the solution, which is based on  Fourier analysis:
$$ u(t,\cdot)= \frac {{\rm
e}^{i \frac {|\cdot|^2} {4t} }} {(4 \pi i t)^\frac {n} 2} \star u_0\, .
$$ 
The dispersive inequality \eqref{dispSR0} which expresses  that  waves with different frequencies move at different velocities,  gives rise  when $u_0$ is in  $L^2(\R^{n})$ to the following Strichartz estimate (see for instance\ccite{bourgain1, bourgain2, bourgain, Caz1, keeltao}) for the solution to the free Schr\"odinger equation
\beq
\label {dispSTR}
\|u\|_{L^q(\R, L^p(\R^{n}))} \leq C(p,q)   \|u_0\|_{L^2(\R^{n})}\,,
\eeq
where $(p,q)$  satisfies the scaling admissibility condition
\beq
\label {admissibR}
\frac 2 q + \frac n p = \frac n 2  \with q \geq 2 \andf   (n,q,p) \neq (2,2, \infty) \, .
\eeq
It is worth noticing that the dispersive inequality~\eqref{dispSR0} also yields  the following Strichartz  inequalities for the inhomogeneous Schr\"odinger equation, which have proven to be of paramount importance in the study of semilinear  and  quasilinear Schr\"odinger equations (one can for instance consult the monograph\ccite{BCD1} and the references therein): 
if~$(p,q)$  and~$(p_1,q_1)$ satisfy the admissibility condition \eqref{admissibR}, then
$$
  \|u\|_{L^q(\R, {L^{p}(\R^{n}))}} \leq C(p,q, p_1,q_1) \Big(    \|u_0\|_{L^{2}(\R^{n})} + \|i\partial_t u -\D u\|_{L^{q'_1}(\R, {L^{{p'_1}}(\R^{n}))}}\Big)\,,
$$
 denoting by~$a'$ the dual exponent of any~$a \in [1,\infty]$\footnote{similar Strichartz estimates hold in a Sobolev framework, with adapted indices.}.
 
\medskip 
\noindent In the case of the wave equation on~$\R^{n}$  \label{waveeqRn}
$$ (W) \qquad  \left\{
\begin{array}{c}
\partial_t^2 u -\D u = 0\\
(u,\partial_t u)_{|t=0} = (u_0,u_1)\,,
\end{array}
\right. 
$$
the solution of which may be written by means of the  Fourier transform~$\mathcal F$ (in the case when~$\mathcal F u_0$  and~$\mathcal F u_1$ are supported in a ring) as $\ds u(t) = \sum_{\pm} u_{\pm}(t)$ with
\begin{equation}
\label{defgammapm}
 u_{\pm}(t) = 
{\mathcal F}^{-1} \big(
e^{\pm it|\xi|}
\gamma_\pm (\xi)\big)\, , \quad \gamma_\pm (\xi)\eqdefa \frac12 \big(
{\mathcal F} u_0(\xi) \pm \frac1{i|\xi|}{\mathcal F} u_1(\xi) 
\big)
\,,
\end{equation}
 the dispersive estimate  writes  (for $t \neq 0$)
\beq
\label {dispSR}
\|u(t,\cdot)\|_{L^\infty(\R^{n})} \leq \frac C {|t|^{\frac {n-1} 2}} ( \|u_0\|_{L^1(\R^{n})}+ \|u_1\|_{L^1(\R^{n})})\,.
\eeq
Its proof   requires more
elaborate techniques involving oscillatory integrals  and    the application of a stationary  phase theorem. This dispersive estimate leads to the  following Strichartz estimate  (see for instance\ccite{BCD1,ginibrevelo,keeltao} and the references therein)
\beq
\label {dispSTR}
\|u\|_{L^q(\R, L^p(\R^{n}))} \leq C(p,q) \big(  \|\nabla u_0\|_{L^2(\R^{n})}  + \|u_1\|_{L^2(\R^{n})} \big) \,,
\eeq
where $(p,q)$  satisfies the scaling admissibility condition
\beq
\label {admissibRwave}
 \frac 1 q + \frac n p = \frac n 2 -1\with p,q \geq 2 \andf  q<\infty\,.
\eeq
If~$(p,q)$  and~$(p_1,q_1)$ satisfy~\eqref{admissibRwave},  one can also infer  
$$
  \|u\|_{L^q(\R, {L^{p}(\R^{n}))}} \leq C(p,q, p_1,q_1) \Big(   \|\nabla u_0\|_{L^2(\R^{n})}  + \|u_1\|_{L^2(\R^{n})}  + \|\partial_t^2 u -\D u\|_{L^{q'_1}(\R, {L^{{p'_1}}(\R^{n}))}}\Big)\,.
$$

\bigskip 
\noindent  When some loss of 
dispersion occurs,    as for instance in the case of compact Riemannian manifolds and of some bounded domains, or as   was highlighted by Bahouri, G\'erard and Xu in\ccite{bgx} in the case of the Schr\"odinger   operator on~$\H^d$ (where it is shown that there is no dispersion at all),  the Euclidean strategy    referred to above fails and the problem  of obtaining Strichartz estimates is considered as very difficult. Strichartz estimates in  the setting of compact Riemannian manifolds  and bounded domains (with a possible loss of derivatives) have been obtained in a number of works (see for instance   Bourgain\ccite{bourgain3},  Burq, G\'erard and  Tzvetkov\ccite{bgt}, Ivanovici,   Lebeau and Planchon\ccite{lebeau} and the references therein). The case of the hyperbolic space (noncompact and negatively curved) is also considered in \cite{anker}.

 \smallskip Even if the study of PDEs associated with sublaplacians on nilpotent groups is nowadays classical (see for instance the pioneering works \cite{gaveau,hul,RS76}), obtaining Strichartz estimates for the Schr\"odinger operator on the Heisenberg group is still open and has to our knowledge never been tackled. Note that the Heisenberg group is one of the simplest examples of a noncommutative Lie group, whence our interest in proving those estimates in this setting. We are confident that our methods should apply to more general nilpotent Lie groups,  provided some harmonic and Fourier analysis tools (that will be introduced in the setting of~$\H^d$ in Section \ref{Fourier}) are extended from the Heisenberg framework to the context of these groups.  This is for instance the case  of  H-type groups  or more generally of  step~$2$ stratified Lie groups: in\ccite{hiero}  and\ccite{bfg}  the lack of 
dispersion for the associated Schr\"odinger operators is indeed proved.  We also refer to \cite{Ber13, Ber17} for a discussion about the link between dispersion, restriction estimates and the heat semigroup  and to \cite{gassot} for a study of the cubic Schr\"odinger equation on the Heisenberg group.

\bigskip
\noindent  In this paper our   main goal is  thus to establish   Strichartz estimates for the solutions to the linear Schr\"odinger equation on    the Heisenberg group~$\H^d$,   involving the sublaplacian, as well as for the wave equation. As already mentioned,   in\ccite{bgx} the authors show
the absence of dispersion -- they actually prove 
  that the Schr\"odinger equation on~$\H^d$ behaves as  a transport equation with respect to one direction, known as the vertical direction (i.e., along the orbits of the Reeb vector field).  But as   will  be   clear later, a salutary fact is that the Schr\"odinger  operator on~$\H^d$ behaves rather well in  the complement to that vertical direction. This enables us  to   derive anisotropic Strichartz estimates for  the  Schr\"odinger operator on~$\H^d$, by adapting   the Fourier transform restriction analysis initiated in\ccite{strichartz} and\ccite{Tomas} in the Euclidean case (see also \cite{Fefferman}); this   also leads to new, anisotropic Strichartz estimates for the wave equation, at least for the radial case.  The approach we set up here is somewhat more  challenging  than in the Euclidean case because  the Fourier analysis on the Heisenberg group is an intricate tool.

\subsection{Basic facts about the  Heisenberg group}\label {introbasic} Let us start by  recalling that the~$d$-dimensional Heisenberg group~$\H^d$ can be defined as~$T^\star\R^d \times\R$   where~$T^\star\R^d$ is  the cotangent bundle,    endowed  with the noncommutative  product law
\begin{equation}\label{lawbis} 
(Y,s)\cdot (Y',s') \eqdefa \big(Y+Y',   s+s'+2 \sigma(Y,Y')\big)\,,
\end{equation}
where\footnote{The variable $Y$ is called the horizontal variable, while the variable $s$ is known as the vertical variable.} $w=(Y,s)=(y,\eta,s)$ and $w'=(Y',s')=(y',\eta',s')$ are  elements of~$\H^d$,  while $\sigma$ denotes the  canonical symplectic form   on~$T^\star\R^d$   defined by
\begin{equation}\label{defsigma} 
\sigma(Y,Y') \eqdefa \langle\eta,y'\rangle-\langle\eta',y\rangle\quad\hbox{for all }\ 
(Y,Y')\in T^\star\R^d \times T^\star\R^d\, ,
\end{equation}
 with $\langle \eta,y\rangle$   the value of the one-form $\eta$  at~$y.$

 \medskip  
 \noindent With this point of view, the  Haar measure on $\H^d$ is simply  the  Lebesgue measure on the space~$T^\star\R^d \times \R$.
  In particular, one can define   the following (noncommutative) convolution product  for any two integrable functions~$f$ and~$g$:
\beq
\label {definConvolH}
f \star g ( w ) \eqdefa \int_{\H^d} f ( w \cdot v^{-1} ) g( v)\, dv 
= \int_{\H^d} f ( v ) g( v^{-1} \cdot w)\, dv\, .
\eeq
Even though the
convolution on the Heisenberg group is noncommutative, if one
 defines the  Lebesgue spaces  $L^p(\H^d)$  to be simply  $L^p(T^\star\R^d \times \R ),$ then
one  still obtains  H\"older and Young inequalities, in their classical and weak versions.  In order to 
distinguish the vertical coordinate from the others, we shall also be using, for any two real numbers~$1\leq p,r \leq \infty$,  the anistropic Lebesgue spaces $L^{p}_{Y}L^{r}_{s}(\H^d)$ and $L^{r}_{s}L^{p}_{Y}(\H^d)$ endowed with the mixed norms
$$
 \|f\|_{L^{p}_{Y}L^{r}_{s}}\eqdefa   \left(\int\left( \int |f(Y,s)| ^r \, ds\right)^\frac pr \, dY\right)^\frac1p\, ,\quad 
 \|f\|_{L^{r}_{s}L^{p}_{Y}}\eqdefa   \left(\int\left( \int |f(Y,s)| ^p \, dY\right)^\frac rp \, ds\right)^\frac1r\, .
 $$
 In the framework of the Heisenberg group, the scale invariance is investigated through the family of  
 dilation operators~$(\d_a)_{a>0}$ (which are compatible with the product law    \eqref{lawbis}) defined by
\begin{equation}
\label{dilation} \delta_a ( Y,  s ) \eqdefa ( a Y,  a^2s )\, .\end{equation} 
As  the determinant of~$\d_a$ is~$a^{2d+2}$, it is natural to  define 
the  homogeneous dimension of~$\H^d$   to be~$Q\eqdefa 2d+2$. 

 \medskip  
\noindent The Schwartz class~$\cS(\H^d)$ coincides with $\cS(\R^{2d+1})$, and can be characterized by the action of the sublaplacian
$$
\D_{\H}u  \eqdefa \sum_{j=1} ^d (\cX_j^2u+\Xi_j^2u) \, ,
$$
 where the horizontal vector fields~$\cX_j$ and~$\Xi_j$ are defined for~$ j\in\{1,\dots,d\}$ by
\begin{equation}
\label {horizvect}
 \cX_j\eqdefa\partial_{y_j} +2\eta_j\partial_s\andf \Xi_j\eqdefa \partial_{\eta_j} -2y_j\partial_s \, .
\end{equation}
We also define the horizontal gradient
$$
 \nabla_{\H}u \eqdefa (\cX_1 u,\ldots,\cX_d u,\Xi_1 u,\ldots,\Xi_d u) \, .
$$
Note that Sobolev spaces can be defined thanks to the sublaplacian, for instance for any real number~$s$
$$
H^s(\H^d)\eqdefa\Big\{u \in L^2(\H^d) \, ,  (-\Delta_\H)^\frac s2 u \in L^2(\H^d)  \Big\}\, ,
$$
where   non integer powers of~$-\Delta_\H$ can be defined via functional calculus. The purpose of this paper is to establish   Strichartz estimates for the linear Schr\"odinger  and wave equations  on $\H ^d$  associated  with the sublaplacian
$$
(S_\H)\quad \left\{
\begin{array}{c}
i\partial_t u -\D_\H u = f\\
u_{|t=0} = u_0\,,
\end{array}
\right. 
\qquad \quad 
(W_\H)\quad \left\{
\begin{array}{c}
\partial_t^2 u -\D_\H u = f\\
(u,\partial_t u)_{|t=0} = (u_0,u_1)\,.
\end{array}
\right. 
$$  
 As in the Euclidean case, among  the  most  notable  achievements  of Fourier analysis on the Heisenberg group that we   review in Section~\ref{Fourier}, one can mention that one can explicitly solve those equations by means of the Fourier transform. 
However as  shown by the following proposition  established
 in\ccite{bgx}, $(S_\H)$   is  a model    for totally non-dispersive evolution equations. 
\begin{proposition} [\cite{bgx}]
\label {nodispersionS}
{\sl
There exists a function~$u_0$ in the Schwartz class~$\cS(\H^{d})$ such that the solution to  the free Schr\"odinger equation $(S_\H)$  (with~$f\equiv 0$) satisfies
$$
u(t,Y,s) = u_0(Y,s+4td)\, .
$$
}
\end{proposition}

 \medbreak
\begin{remark}\label{rksh} {\sl Since the translation $
(Y,s)\mapsto (Y,s+s_0)$ leaves the  Lebesgue measure invariant for all~$s_0 \in \R$, the solution constructed in Proposition~{\rm\ref{nodispersionS}} satisfies
$$
\forall p \in [1,\infty]\,,\ \|u(t,\cdot)\|_{L^p(\H^{d})} = \|u_0\|_{L^p(\H^{d})}  
$$ 
which shows that one cannot hope for a dispersion phenomenon of the type \eqref {dispSR}.}
\end{remark}

\medbreak
\begin{proof}
In order to establish Proposition\refer{nodispersionS}, let us  introduce a family of functions on~$\H^{d}$ which are  the analogues of the solutions associated with plane waves in the classical Euclidean case, namely  
$$
(t,x)\in\R\times\R^n\longmapsto e^{i|\xi|^2 t +i\langle \xi,x\rangle}\in{\mathbb S}^1
$$
which of course satisfy
$$
(i\partial_t -\D)   e^{i|\xi|^2 t +i\langle \xi,x\rangle} =0\,.
$$
Similarly, consider the family of functions 
\beq
\label {spectreDeltaHdemoeq0}
\Theta_\lam: 
(Y,s)\in \H^d \longmapsto e^{is\lam} e^{-\lam |Y|^2}\in \C\,.
\eeq
One can  readily  check that \beq
\label {spectreDeltaHdemoeq1}
-\D_{\H} \Theta_\lam  = 4\lam d \Theta_\lam\,,
\eeq
therefore  the functions
$$
(t,Y,s)\in \R\times\H^d \longmapsto \Theta_\lam(Y,s+4td)\in \C
$$
satisfy
$$
(i\partial_t -\D_\H)  \big(  \Theta_\lam(Y, s+4td)\big) =0\,.
$$
Now let~$g$ be a function  in~$\cD(]0,\infty[)$, and define
$$
u(t,Y,s) \eqdefa \int_{\R} \Theta_\lam(Y,s+4td) g(\lam) \, |\lam|^d d\lam\,.
$$
It stems  from the Lebesgue derivation theorem that $u$ solves the Cauchy problem $(S_\H)$
with~$f \equiv 0$ and initial data
\beq
\label {dataSH}
u_0(Y,s) = \int_{\R} \Theta_\lam(Y,s) g(\lam) \, |\lam|^d d\lam\,,
\eeq
which   easily ends   the proof of the proposition.
\end{proof}

\medbreak

   Actually, as we shall see in Section \ref{Fourier} \pagerefer{newFourierconvorad}, there is a family of functions  $(\Theta^{(\ell)}_\lam)_{\ell \in \N}$ on $\H^d$ such that\footnote{The function $\Theta^{(0)}_\lam$ corresponds to the function $\Theta_\lam$ given by\refeq{spectreDeltaHdemoeq0}.} 
\begin{equation}\label{transportequationell}
(i\partial_t -\D_\H)  \big(  \Theta^{(\ell)}_\lam(Y, s+4t(2\ell+d))\big) =0\,.
\end{equation}
This readily ensures that the solution to  the free Schr\"odinger equation $(S_\H)$ associated to the Cauchy data  
$$
u^{(\ell)}_0(Y,s) \eqdefa \int_{\R} \Theta^{(\ell)}_\lam(Y, s) g(\lam) \, |\lam|^d d\lam\,,
$$
with~$g \in \cD(]0,\infty[)$,  behaves  as a transport equation, with velocity depending on $\ell$. More precisely, we have 
\beq
\label {penalization}
(i\partial_t -\D_\H)  u^{(\ell)}_0= i (\partial_t - 4 (2\ell+d) \partial_s)  u^{(\ell)}_0\,,
\eeq
which again highlights the fact  that one cannot hope for a dispersion phenomenon of the type \eqref {dispSR}.

\begin{remark}\label{rkwaveHd}
{\sl In\ccite{bgx} the authors also  prove that every solution to the wave equation on~$\H^{d}$ satisfies the dimension-independent dispersive estimate
\beq
\label{dispwaveoptimal}
\|u(t,\cdot)\|_{L^\infty(\H^{d})} \leq \frac C {|t|^{\frac12}} ( \|u_0\|_{L^1(\H^{d})}+ \|u_1\|_{L^1(\H^{d})})\,,\eeq
and show by an example similar  to the ones above   that this estimate is optimal.
The rate of decay in{\rm\refeq{dispwaveoptimal}} regardless to the dimension is due to the fact that only the center is involved in the dispersive effect. Note also    that compared with the Euclidean  framework, there is an  exchange in the rates of decay between the wave and the Schr\"odinger equations on $\H^{d}$.
It is also proved in\ccite{bgx} that the dispersive estimate~{\rm(\ref{dispwaveoptimal})} gives rise  to a Strichartz estimate
$$
\|u\|_{L^q_t L^p_{Y,s}} \leq    C_{p, q,p_1,q_1} \Big(   \|\nabla_{\H^d}  u_0\|_{L^2(\H^{d})}  + \|u_1\|_{L^2(\H^{d})}\\
 + \|f\|_{ L^{q_1'}_t L^{p'_1}_{Y,s}}\Big)
$$
with $\displaystyle \frac1q+\frac Qp = \frac Q2 -1$ and~$q \geq 2Q-1$.  
}
\end{remark}
\medbreak

\subsection{Statements of the results}
Our first goal in this paper is to establish the following Strichartz estimates for the Schr\"odinger equation on~$\H^{d}$ for radial data --- note that the Fourier transform in the radial
setting is much easier to handle, and the geometry of sets  on the Fourier side is also much easier to describe in the radial case (see for example~(\ref{defsphere}) in Section~\ref{studysurfacemeasure} for the    sphere), so we restrict our attention to that framework in this article. A function $f$ on ~$\H^{d}$ is said to be \emph{radial} if it is invariant under the action of the unitary group~$U(d)$ of $T^\star\R^d$, which implies that $f$ can be written under the form~$f(Y,s)= f(|Y|, s)$. 
\begin{theorem}
\label{STth}
{\sl
Given $(p,q)$   belonging to  the admissible set 
 $$
\cA^{\mbox{\tiny{S}}}\eqdefa\Big\{(p,q)\in [2,\infty]^2\, /\,p\leq q \quad \mbox{and} \quad\frac2q+\frac{2d}p\leq \frac Q2\Big\} \,,
$$
there is a constant~$ C_{p, q}$ such that the  solution to  the Schr\"odinger equation $(S_\H)$ associated with radial data satisfies the following  Strichartz estimate  (denoting by~$a'$ the dual exponent of any~$a \in [1,\infty]$)
\beq
\label{dispSTH}
\|u\|_{L^\infty_s L^{q}_t L^{p}_{Y}} \leq C_{p, q }  \Big(\|u_0\|_{H^{ \frac Q2 - \frac2q-\frac{2d}p}(\H^{d})} + \|f\|_{L^1(\R, H^{ \frac Q2 - \frac2q-\frac{2d}p}(\H^{d}))} \Big) \,.
\eeq 
} 
\end{theorem}
\begin{remark}\label{rkindices} {\sl The above theorem deserves some comments:
\begin{itemize} 
   \item First notice that the limit case when~$\displaystyle \frac2q+\frac{2d}p= \frac Q2$ provides (when $f=0$) the estimate
$$
  \|u\|_{L^\infty_s L^{2}_t L^{2}_{Y}} \leq C   \|u_0\|_{L^2(\H^d)}
$$
which has the flavor of  the energy equality $  \|u\|_{L^\infty_t L^2(\H^d)} =   \|u_0\|_{L^2(\H^d)}$ but is of course of a different nature.

\smallskip
 \item In the case when~$\displaystyle \frac2q+\frac{2d}p < \frac Q2\virgp$ we have a larger  range of admissible Strichartz pairs, but with a loss of derivatives compared to the euclidean framework. Such phenomena   have been proved to hold in Riemannian compact manifolds as well as in the case of variable-coefficient  evolution equations with low regularity (see for instance\ccite{bch, bch2,  bourgain3, bgt, Tataru} and the references therein).
 
  \smallskip
 \item  Note that the Strichartz estimate\refeq{dispSTH} is invariant by scaling (through the scaling~$u (t,w)\mapsto u (\Lambda^{-2} t, \d_{\Lambda^{-1}}w)   
$ and~$f (t,w)\mapsto \Lambda^{-2}  f (\Lambda^{-2} t, \d_{\Lambda^{-1}}w)   
$.)  
 
 \smallskip
 \item A natural open question is to obtain Strichartz estimates without loss of derivatives for a larger  range of admissible Strichartz pairs. It is difficult to state a conjecture on the  
range of indices for which the estimates should be valid and the heart of the matter relies on the possibility to relax or not the constraint $p\geq 2$. Actually,  as   will be seen later, our approach is inspired by the remarkable paper of M\"uller\ccite{Muller} where a counter-example is provided for $\displaystyle p<p_d:= \frac{4d} {2d+1}\cdotp$  In order to investigate the sharpness of our results, one first needs to have more insight about the restriction result of M\"uller for $p_d\leq  p<2 $ which is a very challenging issue (see Section\refer{restrictionmuller} for further details). 
 \end{itemize} }
\end{remark}

\medbreak

In the case of the wave equation we obtain the following Strichartz estimate.

\begin{theorem}
\label{STthwave}
{\sl
With the above notation, given $(p,q)$    belonging to  the admissible set  
 $$
\cA^{\mbox{\tiny{W}}}\eqdefa\Big\{(p,q)\in [2,\infty]^2\, /\, p \leq q \quad \mbox{and} \quad\frac1q+\frac{2d}p \leq \frac Q2-1\Big\} \,,
$$
there is a constant~$ C_{p, q}$ such that the  solution to  the wave equation $(W_\H)$ associated with radial data satisfies the following  Strichartz estimate:
$$
\begin{aligned}
  \|u\|_{L^\infty_s L^{q}_t L^{p}_{Y}} \leq C_{p, q} \Big( &  \|u_0\|_{H^{ \frac Q2 - \frac1q-\frac{2d}p}(\H^{d})}  + \|u_1\|_{H^{ \frac Q2 - \frac1q-\frac{2d}p- 1}(\H^{d})}+   \|f\|_{L^1(\R, H^{ \frac Q2 - \frac1q-\frac{2d}p - 1}(\H^{d}))}\Big)\,.
\end{aligned}$$
}
\end{theorem}
\begin{remark}\label{rkindicesW} {\sl   Note that  in the limit case when~$\displaystyle \frac1q+\frac{2d}p= \frac Q2 -1\virgp$ the set of admissible Strichartz pairs is wider than for the Schr\"odinger equation. In some sense, this is not surprising since as   was already highlighted in Remark{\rm\refer{rkwaveHd}},  compared with the Euclidean  framework  there is an  exchange in the behavior  between the wave and the Schr\"odinger equations on~$\H^{d}$, and as   is well-known, in the Euclidean case  the Schr\"odinger equation enjoys better dispersive estimates than the wave equation. }
\end{remark}

 \medbreak

 Our strategy of proof of the estimates is closely related to the method developed in~\cite{strichartz} (the reader may consult\ccite{Tao} and the references therein for an overview on this subject in the Euclidean framework, as well as Section~\ref{Erren} below) consisting in reducing the problem to the study of the restriction operator on a manifold in   Fourier space --- with additional non negligible technicalities owing to the complexity of the Fourier transform  on the Heisenberg group.  That is actually the main achievement of this paper. At this stage, one should mention the Fourier restriction theorem on~$\H^{d}$ due to M\"uller (\cite{Muller}), where the author   investigated the restriction of the Heisenberg  Fourier  transform on the   unit sphere and emphasized the separate roles of the horizontal and vertical variables of~$\H^{d}$. 
 
\smallskip  
 {  Our limitation in Theorems \ref {STth} and~\ref {STthwave} to   radial data is intimately  linked to the complexity of the manifolds in the Fourier side outside the radial framework: see Remark~\ref{sphere} for more on this.} 
 
\smallskip 
 Other results extending the restriction theorem of M\"uller to more general nilpotent groups through  spectral analysis have been considered in \cite{CC13, CC17} and \cite{LS11,LS16}. Finally, let us mention that applications of non commutative Fourier analysis have been also used to study the heat equation associated to sublaplacians on groups, see for instance \cite{ABGR09}. For our purposes, we need  Fourier restriction estimates in a direct product of the  Heisenberg group and the  real line, which  will be obtained by combining the methods of M\"uller\ccite{Muller} and Tomas-Stein\ccite{Tomas}.

\medbreak

\subsection{Layout}\label {layout}
The   proof of Theorems \ref {STth} and~\ref {STthwave}
is addressed in Section \ref{proofmainth}.
A short  illustration of the proof in the (well-known) Euclidean case is provided in Section~\ref{Erren} for the convenience of the reader.
The Fourier transform on~$\H^d$ and the space of frequencies~$\wh\H^d$ are defined and described in Section~\ref{Fourier}, 
while Section~\ref{unitsphHt} is dedicated to the study of the restriction of the Heisenberg Fourier transform  to the unit sphere of the frequency space~$\wh\H^d$: this is not strictly necessary to the proof of our main results but will be a way of introducing our methods, by recovering the results of M\"uller~\cite{Muller} in a slightly simpler setting. 
 Finally in the Appendix we recall some properties of $\lambda$-twisted convolutions which  are needed in the proof.

\medskip  
To avoid heaviness, all along this article~$C$ will denote  a  positive  constant   which may vary from line to line.    We also use $f\lesssim g$  to
denote an estimate of the form $f\leq C g$.

\bigbreak\noindent{\bf Acknowledgments.}
  The authors wish to thank   warmly  Jean-Yves Chemin  for  numerous  enlightening discussions.  {They  extend their thanks to the  anonymous referees for a careful reading of the manuscript and very useful remarks, which led to several improvements.}

The second author was supported by the Grant ANR-15-CE40-0018 ``Sub-Riemann\-ian Geometry and Interactions'' of the French ANR. 

\section{Fourier restriction theorem and its applications in the Euclidean space}\label{Erren}
In this section we recall some classical results on the Fourier restriction problem and its application to PDEs in the classical, Euclidean setting for the convenience of the reader, since we shall follow a similar approach in our framework. 
To keep the notation consistent with the case of the Heisenberg group that  follows, we distinguish $\R^{n}$ and its dual $\wh \R^{n}$, which is of course isomorphic to $\R^{n}$ itself. 
\subsection{Restriction theorems}
The Fourier transform $\cF (f)$ of a function $f$ in $L^1(\R^n)$ is continuous,  thus it makes sense to restrict $\cF (f)$  to any subset of $\wh \R^n$. However, the Fourier transform of a function in $L^2(\R^n)$ is, in general, only in $L^2(\wh \R^n)$, hence completely arbitrary on a set~$\wh S$ of $\wh\R^n$ of measure zero. 

Indeed, in general, the Fourier transform of a function in $L^{p}$ for $p>1$ cannot be restricted to an hyperplane. As one can easily check, the function $f:\R^{n}\to\R$ defined by
\beq
\label{cexR}\ds  f(x)= \frac{e^{- |x'|^2}} {1+|x_1|}\, \virgp\qquad x=(x_1,x')\in \R^{n},
\eeq
belongs to $L^p(\R^n)$, for all $p>1$, but its Fourier transform does not admit a restriction on the hyperplane $\wh S$ of $\wh\R^n$ defined by $ \wh S=\{\xi \in \wh  \R^n \, /\, \xi_1=0 \}$. 

Tomas and Stein   made the surprising discovery that one can  restrict the Fourier transform of $L^p(\R^n)$ functions, for $p>1$ (and close to $1$),  to hypersurfaces $\wh S$ that are ``sufficiently curved'', as for instance the sphere.
More   generally, given a hypersurface~$\wh S \subset \wh\R^n$ endowed with a smooth measure $d\sigma$, the restriction problem asks for which pairs $(p,q)$ an inequality of the form
\begin{equation}\label{eq:estime0}
 \|\cF( f)|_{\wh S}\|_{L^q (\wh S, d \sigma)}\leq C \|f\|_{L^p (\R^n)}
\end{equation}
holds for all $f$ in $\cS(\R^n)$. 

Despite all the recent progresses in this field, this question is not completely settled in its general form and remains a topical issue. For a general survey on these questions we refer to the book of Stein \cite{stein} and the text of Tao\ccite{Tao}. 
In what follows, we focus on the case $q=2$.

By a duality argument, the above question for $q=2$ is equivalent to asking whether the adjoint operator $R^{*}_{S}$ defined by
$$R^{*}_{S}g\eqdefa\cF^{-1} (gd\sigma)$$
is continuous from $L^{2}(\wh S, d\sigma)$ to $L^{p'}(\R^n)$, where $p'$ is the dual exponent of $p$. 

A basic counterexample shows that the range of $p$ for which the estimate holds cannot be the entire interval $1\leq p\leq 2$; for details we refer to \cite{strichartz}.
\begin{example}[Knapp] {\sl 
Let $\wh S$ be the $(n-1)$-dimensional sphere in $\wh \R^{n}$ endowed with the standard measure $d\mu$. Let $g_{\delta}$ be the characteristic function of a spherical cap
$$\wh C_{\delta}\eqdefa\{x\in \wh S : |x\cdot e_{n}|<\delta \}\,.$$
With some computation one can prove that as~$\delta \to 0$,  $$\|g_{\delta}\|_{L^{2}(\wh S,d\mu)}\sim  \delta^{(n-1)/2},\qquad \| \cF^{-1}(g_{\delta}) \|_{L^{p'}( \R^{n})}\geq C \delta^{n-1}\delta^{-(n+1)/p'}\,,$$ hence the estimate can hold only if $p'\geq (2n+2)/(n-1)$, i.e., if $p\leq (2n+2)/(n+3)$.
}
\end{example}
The above range is indeed the correct one in the case of a surface with non vanishing curvature. This is the statement of the so-called Tomas-Stein theorem.
\begin{theorem}[\cite{Tomas}] \label{t:tsbase} 
{\sl Let $\wh S$ be a smooth compact hypersurface in $\wh \R^{n}$ with non vanishing Gaussian curvature at every point, and let $d\sigma$ be a smooth measure on $\wh S$. Then  there holds for every $f\in \cS(\R^n)$ and every $p\leq (2n+2)/(n+3)$,
$$
 \|\cF (f)|_{\wh S}\|_{L^2 (\wh S, d \sigma)}\leq C_p \|f\|_{L^p (\R^n)}\, .
$$}
\end{theorem}
A similar result is possible for surfaces with vanishing Gaussian curvature (that are not flat). In this case the range of $p$ is smaller depending on the order of tangency of the surface to its tangent space. The assumption about compactness of $\wh S$ can be removed by replacing $d\sigma$ with a compactly supported smooth measure.
 
  \subsection{Application of restriction theorems to some PDEs}\label {restRd}
  
Restriction estimates have several applications, from spectral theory to number theory. Here we recall some of these to PDEs: indeed, the restriction theorem can be efficiently applied to obtain Strichartz  estimates on the solutions to some PDEs. Here we focus on the Schr\"odinger and wave equations, for which these estimates were first discovered by Strichartz in his seminal work~\cite{strichartz}.

  \bigskip
  \noindent
Let us first consider the classical  Schr\"odinger equation $(S)$  in $\R^{n}$, recalled in the introduction page~\pageref{schroeqRn}.
Given a solution $u(t,x)$ of this equation, the Fourier transform $\widehat u(t,\xi)$ with respect to the spatial variable $x$ satisfies
\begin{equation}
i\partial_{t}\widehat u(t,\xi)=-|\xi|^{2} \widehat u(t,\xi), \qquad \widehat u(0,\xi)=\widehat{u}_{0}(\xi).
\end{equation}
Solving the corresponding ODE and taking the inverse Fourier transform one has
\begin{equation}\label{eq:str0}
u(t,x)=\int_{\wh \R^{n}}e^{i(x\cdot \xi + t|\xi|^{2})}\widehat{u}_{0}(\xi)d\xi\, .
\end{equation}
Formula \eqref{eq:str0} can be interpreted as the restriction of the Fourier transform on the paraboloid~$\wh S$  in the space of frequencies $\wh \R^{n+1}=\wh \R \times \wh \R^{n}$,  defined as
$$\wh S\eqdefa\Big\{(\alpha,\xi)\in \wh \R\times \wh\R^{n}\mid \alpha=|\xi|^{2} \Big\}\, .$$
Let us endow $\wh S$ with the measure $d\sigma=d\xi$ induced by the projection $\pi: \wh \R\times \wh \R^{n}\to \wh \R^n$ onto the second factor. More formally one should write\footnote{Given $T:M\to N$ and $\mu$ measure on $M$ we can define a measure $T_{\sharp}\mu$ on $N$ as $T_{\sharp}\mu(A)=\mu(T^{-1}(A))$.} $d\sigma=(\pi|_{\wh S}^{-1})_{\sharp}d\xi$. Notice that~$\pi|_{\wh S}$ is invertible and $d\sigma$   is not the intrinsic surface measure of $\wh S$, which is written in coordinates as $d\mu=\sqrt{1+2|\xi|}d\xi$.

\medskip

Given $\widehat{u}_{0}:\wh \R^{n}\to \C$ define $g:\wh S \to \C$ as $g=\wh{u}_{0}\circ \pi|_{\wh S}$. In other words $g(|\xi|^{2},\xi)=\widehat{u}_{0}(\xi)$.
By construction, for $\wh u_{0}\in L^{2}(\wh \R^{n})$ one has $g\in L^{2}(\wh S,d\mu)$ and $\|u_{0}\|_{L^{2}(\wh \R^{n})}=\|g\|_{L^{2}(\wh S,d\mu)}$.
Then
$$u(t,x)=\int_{\R^{n}}e^{i(x\cdot \xi + t|\xi|^{2})}\widehat{u}_{0}(\xi)d\xi = \int_{\wh S}e^{i y\cdot z} g(z)d\sigma(z)
$$
where $y=(t,x)$ and $z=(\alpha,\xi)$. Theorem~\ref{t:tsbase}, in dual form, tells us that
\begin{equation} 
\| \cF^{-1}(gd\sigma) \|_{L^{p'}(\wh \R^{n+1})}\leq C_p \, \|g\|_{L^{2}(\wh S,d\mu)}\, ,
\end{equation}  
for all  $g\in L^{2}(\wh S,d\mu)$ and all $p' \geq 2(n+2)/n$ (we stress that we apply the result in dimension~$n+1$, i.e., in $\R\times \R^{n}=\R^{n+1}$).

Hence applying the statement to~$g$ related to a initial data $u_{0}$ such that $\widehat{u}_{0}$ is supported on a unit ball (which can be translated in a compact support for $d\sigma$) one has by the Plancherel formula
\begin{equation}\label{eq:str1}
\|u\|_{L^{p'}(\R^{n+1})}\leq C\|u_{0}\|_{L^{2}(\R^{n})}\, ,
\end{equation}
for all $p' \geq 2(n+2)/n$.
 
\medskip
A scaling argument and the density of spectrally localized functions in~$L^2(\R^{n})$, give for~$p'= 2+\frac{4}{n}$ and all~$u_0 \in L^{2}(\R^{n})$
\begin{equation}\label{eq:strlambda00}
\|u\|_{L^{\frac{2n+4}{n}}(\R,
 L^{\frac{2n+4}{n}} (\R^{n}))}\leq C \|u_{0}\|_{L^{2}(\R^{n})}\, .
\end{equation}

\medskip
One can similarly prove a Strichartz estimate for the wave equation $(W)$ in the Euclidean space recalled on page~\pageref{waveeqRn},
by using the   representation formula~(\ref{defgammapm}). The solution can be seen as a sum of two parts, each of which is the restriction of the Fourier transform on one of the two halves of the cone 
 $$\wh S_{\pm}\eqdefa \Big\{(\alpha,\xi)\in \wh \R\times \wh\R^{n}\mid \alpha^{2}=|\xi|^{2}, \pm\alpha>0 \Big\}\, ,$$
each of which endowed with the measure defined by the projection $\pi: \wh \R\times \wh \R^{n}\to \wh \R^n$ onto the second factor (cf.\ the discussion above).

Now let us first assume that~$\gamma_{\pm}$ is frequency localized in a unit ring~$\mathcal C_1$ centered at zero. Then for any $p' \geq 2(n+2)/n$ we have  $$ 
 \|u\|_{L^{p'}(\R^{n+1})} \leq C \| {\mathcal F}_\H^{-1}\gamma_{\pm} \|_{L^{2}(\R^n)} \,.
  $$
As above, for $p'=(2n+2)/(n-2)$, we conclude by scaling arguments and the density  in~$
L^2(\R^n)$ of functions whose Fourier transform is compactly supported in rings centered at zero.

\begin{remark} \label{r:erren} {\sl Notice that to apply the Fourier restriction to evolution PDEs and obtain Strichartz estimates, one  applies the result to a surface in the space $\R^{n+1}=\R\times \R^{n}$, namely the paraboloid and the cone for the Schr\"odinger and wave equation, respectively. 

When dealing with equations defined on the Heisenberg group $\H^{d}$, one is naturally lead to consider surfaces in the space $\R\times \H^{d}$, which is not equal to $\H^{d'}$ for some $d'$. Hence it is not enough to know restriction theorems in $\wh \H^{d}$ (cf.\ Section~{\rm\ref{unitsphHt}}) but one   needs  to adapt these results to surfaces in $\wh \R\times \wh \H^{d}$ (cf.\ Section~{\rm\ref{proofmainth}}).
}
\end{remark}

\section{Fourier analysis on $\H^d$}\label{Fourier}

\subsection{The Fourier transform on $\H^d$}\label {fouriertransf}
As the Heisenberg group  is  noncommutative, defining the Fourier transform of integrable functions on $\H^d$
by means of characters is not   relevant. The standard way consists 
 in using irreducible  representations  of $\H^d$, and in that case the Heisenberg Fourier transform~${\mathfrak{F}}_\H f  (\lambda)$ is not 
a complex valued function on some ``frequency space" as in the Euclidean case, but a family of bounded operators on~$L^2(\R^d)$ (see Corwin and Greenleaf~\cite{corwingreenleaf} for instance for more details). 
Starting from 
the so-called  Schr\"odinger representation, in\ccite{bcdh} and\ccite{bcdh2}  the authors  introduce an alternative definition of the Fourier transform on~$\H^d$ in terms of functions acting on some frequency set $\wt\H^d$. This point of view (which turns out to be equivalent to the classical definition) consists in defining the Fourier transform of an integrable function $f$ on $\H^d$ by projecting~${\mathfrak{F}}_\H(\lam)$ onto the orthonormal basis of~$L^2(\R^d)$  given by   Hermite functions.
This enables to see the Fourier transform of a function $f$ in $L^1(\H^d)$  as the  mean value of~$f$ modulated by some oscillatory functions in the following way:
 \beq
\label {defFH}
\cF_\H f(\wh w) \eqdefa \int_{\H^d}  \overline{e^{is\lam} \cW(\wh w,Y)}\, f(Y,s) \,dY\,ds\,,
 \eeq
for any $\wh w\eqdefa(n,m,\lambda)$ in $\wt\H^d\eqdefa\N^{2d}\times \R\setminus\{0\}$, with $\cW$ the Wigner transform of the (renormalized) Hermite functions
 \beq
\label {defW}
\cW(\wh w,Y)
\eqdefa\int_{\R^d} e^{2i\lam\langle \eta,z\rangle} H_{n,\lam}(y+z) H_{m,\lam} (-y+z)\,dz\,.
 \eeq
Here~$H_{m,\lam}$ stands for the renormalized Hermite function on~$\R^d$, namely~$H_{m,\lam} (x)\eqdefa |\lam|^{\frac d 4} H_m(|\lam|^{\frac 12} x)$, with~$ \suite H m {\N^d}$ the Hermite orthonormal basis of~$L^2(\R^d)$ given by the eigenfunctions of the harmonic oscillator:
$$
-(\D -|x|^2) H_m= (2|m|+d) H_m\, ,
$$
specifically \beq
\label {hermite}
H_m \eqdefa  \Bigl(\frac 1 {2^{|m|} m!}\Bigr) ^{\frac 12} \prod_{j=1}^d  \big(-\partial_j H_0+ x_jH_0\big)^{m_j}  \, ,
\eeq
with $H_0(x)\eqdefa \pi^{-\frac d 4} e^{-\frac {|x|^2} 2}$,  $m!\eqdefa m_1!\dotsm m_d!\,$ and $\,|m|\eqdefa m_1+\cdots+m_d.$

\bigskip In this setting, the  classical 
     statements of Fourier analysis 
hold   in a similar way to the Euclidean case, namely the inversion and Fourier-Plancherel formulae read
\beq
\label {inverseFourierH}f(w) = \frac {2^{d-1}}  {\pi^{d+1} }   \int_{\wt \H^d} 
e^{is\lam} \cW(\wh w, Y)\cF_\H f(\wh w) \, d\wh w 
 \eeq
and
 \beq
\label {FPH}
(\cF_\H f|\cF_\H g)_{L^2(\wt \H^d)}  = \frac {\pi^{d+1}} {2^{d-1}} (f|g)_{L^2(\H^d)}\, ,
\eeq
with  the notation
\beq
\label {measurefrequency}
\int_{\wt \H^d} \theta (\wh w)\,d\wh w \eqdefa \int_{\R} \sum_{(n,m)\in \N^{2d} }\theta(n,m,\lam) |\lam|^d\,d\lam\,.
\eeq
By straightforward computations we find that 
    \beq
\label {imprelW}
-\D_\H \bigl(e^{is\lam} \cW(\wh w,Y)\bigr) =   4|\lam| (2|m|+d)  e^{is\lam} \cW(\wh w,Y) \, ,\eeq
 for any~$\wh w=(n,m,\lambda)$  in $\wt\H^d$, which readily implies that 
$$
\cF_\H( \Delta_\H f) (\wh w) =-4| \lam | (2|m|+d)\cF_\H( f)(\wh w)\,.
$$
This formula allows to give a definition of a function whose Fourier transform is compactly supported, in the following way.
\begin{definition} \label{freqlocgen}
 {\sl We say that a   function~$f$ on~$\H^d$ is frequency localized in a ball~${\mathcal B}_\Lambda$ centered at 0  of radius~$\Lambda$ if there exists an  even  function~$\psi$ in $\cD(\R)$   supported in~${\mathcal B}_1$  and equal to~1 near 0 such that\footnote{where~$\psi(-\D_\H) $ is defined by the functional calculus of the self-adjoint operator~$-\D_\H$.}  
$$
f  = \psi(- \Lambda^{-2}\D_\H) \, f \, ,   
$$
which is equivalent to stating that  for any $\wh w=(m,n,\lam)$ in $\wt {\H}^d$,
$$
 \cFH (f)(n,m,\lam) = \psi (\Lam^{-2} 4|\lam|(2|m|+d))  \, \cFH (f)(n,m,\lam)\, .
$$
Similarly we say that a   function~$f$ on~$\H^d$ is frequency localized in a ring~${\mathcal C}_\Lambda$ centered at 0  of  small radius~$\Lambda/2$ and large radius~$\Lambda$ if there exists an  even  function~$\phi$ in $\cD(\R)$   supported in~${\mathcal C}_1$  and equal to 1 in a ring~${\mathcal C}'$ contained in~${\mathcal C}_1$ such that
$$
f  = \phi(- \Lambda^{-2}\D_\H) \, f \, ,   
$$
which is equivalent to stating that  for any $\wh w=(m,n,\lam)$ in $\wt {\H}^d$,
$$
 \cFH (f)(n,m,\lam) = \phi (\Lam^{-2} 4|\lam|(2|m|+d))  \, \cFH (f)(n,m,\lam)\, .
$$
 }
\end{definition}
One of the interests of this definition lies in the following proposition, whose proof may be found in\ccite{bahourigallagher} and\ccite{bgx}.
 \begin{lemma}
\label  {bernstein}
{\sl  With the above notation,  
\begin{itemize}
\item if~$f$  is frequency
localized in~${\mathcal B}_\Lambda$, 
  then for all $1 \leq p \leq q \leq \infty$,      $k\in\N$ and $\beta\in\N^{2d}$ with $|\beta|=k,$  there is a constant~$C_k$ depending only on~$k$ such that
\begin{equation}
\label{eq:lech1}  \norm{\cX^\beta f}{L^q(\H^d)}\leq C_k \Lam^{k+Q (\frac{1}{p}-\frac{1}{q}) } \norm
f{L^p(\H^d)} \, ,
\end{equation}
where $\cX^\beta$ denotes a product of $|\beta|$ vectors fields of type \eqref{horizvect};

\item if~$f$  is frequency
localized in~$\cC_\Lambda$,  then for all $ p \geq 1$ and     $s\in\R$, there is a constant~$C_s$ depending only on~$s$ such that
 \begin{equation}
\label{eq:lech3} 
 \frac 1 {C_s} \Lam^{s} \norm f {L^p(\H^d)} \leq
 \norm{(-\Delta_\H)^{\frac s 2}  f}{L^p(\H^d)} \leq C_s
\Lam^{ s} \norm f{L^p(\H^d)}\, \cdot
\end{equation}
 \end{itemize}}
\end {lemma}
It will be useful later on to  observe  that for any function~$f$ in~$L^1(\H^d)$ and any positive real number~$a$, there holds
\begin{equation}\label {newFourierconvoleq-1}
\forall \wh w=(n,m,\lambda)\in\wt\H^d\,,\  \cFH (f\circ \d_a) (\wh w) =   a^{-Q}\cFH(f)(n,m,a^{-2}\lam)\,.
\end{equation}
Let us also emphasize that if~$f$ and~$g$ are two functions of~$L^1(\H^d)$ then for any~$\wh w=(n,m,\lambda)$ in $\wt\H^d,$
\beq
\label {newFourierconvoleq1}
 \cFH (f\star g) (\wh w)   = ( \cF_\H f \cdot \cF_\H g)(\wh w)\eqdefa \sum_{p\in \N^{d}} \cF_\H f(n,p,\lam)\cF_\H g(p,m,\lam)\,.
 \eeq
In the radial framework (recall that~$f$ is {radial} if it is invariant under the action of the unitary group~$U(d)$  of $T^\star\R^d$), which is our concern in this paper,   it turns out  that    for any function $f$ in~$L_{\rm {rad}}^1(\H^d)$ there holds
 \beq
\label {newFourierrad} \cFH (f) (n, m, \lam)=  \cFH (f) (n, m, \lam)\delta_{n,m}= \cFH (f) (|n|, |n|, \lam) \delta_{n,m} \,.
\eeq
The interested reader can consult for instance\ccite{bfg2, farautharzallah, nach}.
Actually  the Fourier transform~$\cFH$ acts  in the following way  on radial functions:
$$
\cFH (f) (\ell, \ell, \lam) =  \begin{pmatrix} \ell +d-1 \\ \ell \end{pmatrix}^{-1}\int_{\H^d}  \overline{e^{is\lam} \wt \cW(\ell,\lam,Y)} f(Y,s) \,dY\,ds\, ,
$$
 with  (see for example\ccite{bfg2, Tricomi, nach} for further details) 
\begin{equation}
\label{princident} \wt \cW(\ell,\lam,Y) \eqdefa \sumetage {n\in\N^d} {|n|= \ell} \cW(n,n,\lam,Y ) =e^{- |\lam||Y|^2 } L_\ell^{(d-1)} ( 2  |\lam|  |Y|^2)\, ,
 \end{equation}
where  $L_\ell^{(d-1)}$ stands for  the  Laguerre polynomial of order $\ell$ and type $d-1$ given for $x\geq 0 $ by$$ 
L_\ell^{(d-1)}(x)\eqdefa \sum^\ell_{k=0} (-1)^k \begin{pmatrix} \ell +d-1\\ \ell -k \end{pmatrix} \frac{x^k}{k!}  \, \cdot
$$
Note that the family of functions~$(\Theta_\lambda^{(\ell)} )_{\ell \in {\mathbb N}}$ mentioned in the introduction of this paper, satisfying  the transport equation~(\ref{transportequationell}), 
 is defined  by the formula
$$
\Theta_\lambda^{(\ell)} (Y,s)  \eqdefa e^{is\lam} \wt \cW(\ell,\lam,Y)\, .
$$
Equation~(\ref{transportequationell}) then follows simply from the fact that
$$
-\Delta_{\mathbb H} \Theta_\lambda^{(\ell)}  = 4|\lambda| (2\ell + d) \Theta_\lambda^{(\ell)}\, .
$$
 Obviously the inversion and Fourier-Plancherel formulae write in that case
$$
f(w) = \frac {2^{d-1}}  {\pi^{d+1} }  \sum_{\ell \in \N} \int_{\R} 
e^{is\lam} \wt \cW(\ell, \lam, Y)\cFH (f) (\ell, \ell, \lam)\, |\lam|^d\,d\lam
$$
and
 $$
 (f|g)_{L^2(\H^d)}= \frac {2^{d-1}}  {\pi^{d+1} } \sum_{\ell \in \N} \begin{pmatrix} \ell +d-1 \\ \ell \end{pmatrix} \int_{\R} 
\cFH (f) (\ell, \ell, \lam)  \, \overline{\cFH (g) (\ell, \ell, \lam)} \, |\lam|^d\,d\lam   \, .
$$
Moreover since for any element $R$ of~$U(d)$,  the
 automorphism~$\theta_R$ of~${\H}^d$ defined by
$$\theta_R(Y,s) \eqdefa (R(Y), s)$$
   preserves the Haar measure
 of~${\H}^d$, we have
 \[ (f \star g)\circ \theta_R= (f\circ \theta_R)\star(g\circ \theta_R)\, ,\]which implies that the space $L^1_{\rm rad}({\H}^d)$ equipped with its standard structure of linear space and with the convolution product  is a  commutative sub-algebra of~$L^1({\H}^d)$.  We deduce that in this framework, (\ref{newFourierconvoleq1}) reduces to
 \beq
\label {newFourierconvorad}
 \cFH (f\star g) (\ell, \ell, \lam)    = \cF_\H f(\ell, \ell, \lam) \cF_\H g(\ell, \ell, \lam) \,.
 \eeq
Finally it    will be important  to observe    that  there holds for all~$w=(Y,s)$ in~$\H^d$,  in the radial setting,
\begin{equation}\label{eq:fouriertranslate}
\sumetage {n\in\N^d} {|n|= \ell} \cF_{\H}(f\circ \tau_w)(n, n, \lam)=\cF_{\H}(f)(\ell, \ell, \lam) e^{-is\lam} e^{ -|\lam||Y|^2 } L_\ell^{(d-1)}(2|\lam||Y|^2 )\, ,
\end{equation} 
where $\tau_{w}$ denotes the left translate defined by $\tau_{w}(w')\eqdefa w \cdot w'$.

\subsection{Frequency  space  for the Heisenberg group}\label {frequency} 
In\ccite{bcdh}, the authors show that the following  distance $\wh d$ on $\wt\H^d=\N^{2d}\times \R\setminus\{0\}$
\beq
\label {defindistancewtH}
\wh d(\wh w,\wh w') \eqdefa  \bigl|\lam(n+m)-\lam'(n'+m')\bigr|_1 +\bigl |(n-m)-(n'-m')|_1+d|\lam-\lam'| \, ,
\eeq
where~$|\cdot|_1$ denotes the~$\ell^1$ norm on~$\R^d$, is appropriate and that the completion of the set~$\wt \H^d$  for this distance is the set
    $$
    \wh \H^d\eqdefa  \bigl(\N^{2d} \times\R\setminus\{0\}\bigr)
    \cup \wh \H^d_0 \with \wh \H^d_0 \eqdefa \{(\dot x, k) \in {\R_{\mp}^d}\times \Z^d \} 
\andf
{\R_{\mp}^d}\eqdefa  (\R_-)^d\cup (\R_+)^d\,.
$$
 It readily stems from\refeq{defFH} that the following continuous embedding  holds:
\beq
\label {embtH}
\cF_\H: L^1(\H^d)  \hookrightarrow L^\infty(\wh \H^d)\, .\eeq Combining the Fourier-Plancherel formula \eqref{FPH} together with   interpolation theory,  we deduce  that,  for all~$1 \leq p \leq 2$,  the Hausdorff-Young inequality holds
$$
\|\cF_\H f\|_{L^{p'}(\wh \H^d)}\leq \|f\|_{L^{p}(\H^d)}\, , $$
 where  $p'$  is the dual exponent of $p$. 
 
\medskip

This new approach enabled the authors in\ccite{bcdh2}  to extend $\cF_\H$ to $\cS'(\H^d)$,  the set of tempered distributions: note that since the Schwartz class~$\cS(\H^d)$ coincides with $\cS(\R^{2d+1})$ then similarly~$\cS'(\H^{d})$  is noting else than $\cS'(\R^{2d+1})$. As in the Euclidean case, this extension is done by duality and the starting point is the characterization of~$\cS(\wh \H^d)$  as the range\footnote{We refer to\ccite{bcdh2} for   the definition of~$\cS(\wh \H^d)$.} of~$\cS(\H^d)$ by $\cF_\H$. Actually  in\ccite{bcdh2}, the authors  prove that the Fourier transform~$\cF_\H$ is a bicontinuous isomorphism between the spaces~$\cS(\H^{d})$ and~$\cS(\wh \H^d),$ and that  the map~$\cF_{\H}$ can be continuously extended from~$\cS'(\H^d)$ into~$\cS'(\wh\H^d)$ in the following way:
\beq
\label {defS'}
\cF_{\H} : \left \{
\begin{array}{ccl}
\cS'(\H^d) & \longrightarrow & \cS'(\wh \H ^d) \\
T & \longmapsto &  \Bigl[\theta \mapsto  \langle T,{}^t\!\cF_{\H}\theta\rangle_{\cS'(\H^d)\times \cS(\H^d)}\Bigr]\, ,
\end{array}
\right.
\eeq
where \beq
\label {linetFF-1}
 {}^t\cF_{\H}\theta(y,\eta,s) \eqdefa \frac {\pi^{d+1}} {2^{d-1}} (\cF_{\H}^{-1}\theta)(y,-\eta,-s) \, .
\eeq
Let us also emphasize that if $T$ is a tempered distribution on $\H^d$,  then  for all~$f$   in~$\cS(\H^d)$  and all~$w$ in~$\H^d,$
\begin{equation}\label{eq:defconv}
(T\star f)(w)=\langle T,\check f\circ\tau_{w^{-1}}\rangle_{\cS'(\H^d)\times\cS( \H^d)} \, ,
\end{equation} 
where $\check f(w)\eqdefa f(w^{-1})$.  

\section{A restriction theorem on the sphere in the frequency space $\wh\H^d$}\label{unitsphHt}

Our purpose here is to  recover  a Fourier  restriction  result on the sphere of~$ \wh\H^d$ due to M\"uller\ccite{Muller}. Our approach is rather different to\ccite{Muller} as we use the Fourier transform as a key tool in obtaining a representation of the Fourier restriction operator, whereas M\"uller uses a spectral representation. As   will be seen in Paragraph \ref {restSchr1} in the proof of    Strichartz estimates for the wave and Schr\"odinger  operators   on $\H ^d$, the interest of our approach  is that it can easily be applied to   more general frameworks.

\subsection {Study of the surface  measure on the sphere of the frequency space}\label{studysurfacemeasure}
 The aim of this section is to recover  the Heisenberg Fourier transform restriction result of M\"uller in\ccite{Muller}. To this end,  let us  start by introducing $\S_{\wh\H^d}$ the unit sphere on $\wh\H^d$:  denoting by~$\wh 0$ the origin   of $\wh\H^d$  (that is the point of $\wh\H^d$ corresponding to~$(\dot x, k)=(0,0)$, with the notation of Paragraph~\ref{frequency}), the sphere  of~$\wh\H^d$ centered 
at the origin    with radius $1$ 
 is defined by 
\beq
\label {defsphere}\S_{\wh\H^d}\eqdefa \Bigl\{(n,n,\lambda)\in\wh\H^d\,/\,
(2|n|+d)|\lambda|=1\Bigr\}\,
\bigcup \,\Bigl\{(\dot x,0)\in\wh\H^d_0\,/\, |\dot x|_1=1\Bigr\}\, ,
\eeq
and  the surface measure $\ds d\sigma_{\S_{\wh\H^d}}$ is given  for all $\theta$  in $\cS(\wh \H^d)$ by the following formula: 
\beq
\label {definmeasure}
\langle d\sigma_{\S_{\wh\H^d}},\theta\rangle_{\cS'(\wh\H^d)\times\cS(\wh \H^d)} 
= \! \sum_{n\in\N^d} \frac1{(2|n|+d)^{d+1}}\Big(\theta\big(n,n,\frac{1}{2|n|+d}\big)
+\theta\Big(n,n,\frac{-1}{2|n|+d}\Big)\Big)\,. 
\eeq  
 We observe that one can show that the measure of  $\wh \H^d_0$ with respect to  $d\wh w$   is zero, and thus in all that  follows,  we shall agree that the measure $d\wh w$ has been extended 
 by $0$ to  the whole of~$\wh\H^d_0,$  and    we shall keep the same notation~$d\wh w$ for the measure on   the whole of~$\wh\H^d$. 
\medbreak

More generally, if $\S_{\wh\H^d}({\sqrt R})$ denotes the sphere  of $\wh\H^d$ centered 
at the origin $\wh 0$ of  radius~${\sqrt R}$,  let us prove that for all $\theta$ in $\cF_{\H}(\cS_{\rm {rad}}(\H^d))$,  there holds 
\begin{multline} \label{dsigmaR}
\langle d\sigma_{\S_{\wh\H^d}({\sqrt R})},\theta\rangle_{\cS'(\wh\H^d)\times\cS(\wh \H^d)} 
\! \\= \! \!\sum_{n\in\N^d}\! \frac {R^{d}}{(2|n|+d)^{d+1}}\Big(\theta\big(n,n,\frac{R}{2|n|+d}\big)
+ \theta\big(n,n,\frac{-R}{2|n|+d}\big)\Big)\,.
\end{multline} 
We start  indeed with the   general formula $$
\int_{\wh\H^d}\theta(\wh  w)\,d\wh w=\int_0^\infty
\biggl(\int_{\S_{\wh\H^d}({\sqrt R})}\theta(\wh w)\,d\sigma_{\S_{\wh\H^d}}({\sqrt R})\biggr)\,dR\, ,
$$
for  all $\theta$ in  $\cF_{\H}(\cS_{\rm {rad}}(\H^d))$.  In view of  \eqref{measurefrequency}
 and \eqref{newFourierrad}, we have
 \begin{equation}\label{dublev}
 \int_{\wh\H^d}\theta(\wh w)\,d\wh w=\sum_{n\in\N^d}\int_\R\theta(n,n,\lambda)\,|\lambda|^d d\lambda\, ,
 \end{equation}
 which thanks to the Fubini
 theorem and  the change of variable $R=(2|n|+d)|\lambda|$  yields
\begin{multline*}  
 \sum_{n\in\N^d}\int_\R\theta(n,n,\lambda)\,|\lambda|^d d\lambda \\= \sum_{n\in\N^d} \int_0^\infty  \frac{R^{d}}{(2|n|+d)^{d+1}}\,\Big(\theta\Big(n,n,\frac{R}{2|n|+d}\Big)
+\theta\Big(n,n,\frac{-R}{2|n|+d}\Big)\Big) 
dR  \, ,
\end{multline*}  
which proves~(\ref{dsigmaR}).
\begin{remark}\label{sphere} {\sl 
It is important here to be in the radial framework, as in that case the distance defined in~{\rm(\ref{defindistancewtH})}  reduces to
$$
\wh d(\wh w,\wh w') =  2 \bigl| \lam n -\lam'  n'\bigr|_1  +d|\lam-\lam'| \, .
$$
In the non radial case, for  $R>1$ then  the sphere $\S_{\wh\H^d}(R)$ (of center ~$\wh 0$ and radius $R$) is a much more complex set than the unit sphere $\S_{\wh\H^d}$. For instance, when $1<R<2$,    one  can easily check that $\S_{\wh\H^d}(R)$ admits $d+1$ connected components in~$\wt \H^d$. That is the reason why we focus here on the radial framework.}
\end{remark}

In order  to investigate boundedness properties of the restriction of $\cF_\H$ to  $\S_{\wh\H^d}$,
we shall   adapt  the Euclidean proof due to Tomas-Stein (see\ccite{Tomas}, and Section~\ref{Erren} of this paper).  To this end, let us first compute~$ \cF^{-1}_\H(d\sigma_{\S_{\wh\H^d}})$.  By definition,    the tempered distribution
$$G
\eqdefa \cF^{-1}_\H(d\sigma_{\S_{\wh\H^d}})
$$ satisfies for all $\theta$  in~$\cS(\wh \H^d)$
$$
\langle d\sigma_{\S_{\wh\H^d}},\theta\rangle_{\cS'(\wh\H^d)\times\cS(\wh \H^d)}=\langle G,{}^t\!\cF_{\H}\theta\rangle_{\cS'(\H^d)\times \cS(\H^d)}  \, .$$
Let us prove the following proposition. 
\begin{proposition}
\label {deffouriereS}
{\sl With the above notation, $G $ is the bounded function on~$\H^d$ defined by
\beq
\label {defG} G(Y,s)=   \frac {2^d} {\pi^{d+1}}  \sum_{n\in\N^d} \frac1{(2|n|+d)^{d+1}} \cos\Big(\frac{s}{2|n|+d}\Big) \cW\Big(n,n,1,\frac{Y}{\sqrt{2|n|+d}}\Big)\,\cdot \eeq }\end{proposition} \begin{proof} 
According to  \eqref{linetFF-1} and to the fact that the Fourier transform~$\cF_\H$ is a bicontinuous isomorphism between the spaces~$\cS(\H^{d})$ and~$\cS(\wh \H^d),$
we have (with $\theta\eqdefa\cF_\H f$) 
\beq
\label {defFouriermeassphere} \langle d\sigma_{\S_{\wh\H^d}},\theta\rangle_{\cS'(\wh\H^d)\times\cS(\wh \H^d)} = \frac {\pi^{d+1}} {2^{d-1}} \langle G, \wt f\rangle_{\cS'(\H^d)\times\cS( \H^d)}  \, ,
\eeq
with $\wt f (y,\eta,s) \eqdefa f(y,-\eta,-s)$.

\medskip Now observe that 
\beq
\label {intform} \theta\Big(n,n,\frac{1}{2|n|+d}\Big)= \int_{\H^d}  e^{-i\frac{s}{2|n|+d}} \,\cW\Big(n,n,1,\frac{Y}{\sqrt{2|n|+d}}\Big)\, f(Y,s) \,dY\,ds \, . \eeq
Indeed, we have by easy computations
$$
\cW\Big(n,n,\frac{1}{2|n|+d}\virgp Y\Big)=\cW\Big(n,n,1,\frac{Y}{\sqrt{2|n|+d}}\Big)\, ,
$$
which gives rise to \eqref{intform}  according to the definition of $\cF_\H$ and to the fact that for all~$\lam $ in~$ \R \setminus\{0\}$, the function  $\cW (n, n, \lam, Y )$ is real valued. 

\medskip Besides by an obvious change of variable, one  also has
 $$
 \cW\Big(n,n,-1,\frac{Y}{\sqrt{2|n|+d}}\Big)=\cW\Big(n,n,1,\frac{Y}{\sqrt{2|n|+d}}\Big) \, \virgp
 $$
which implies that \beq
\label {intform2} \theta\Big(n,n, - \frac{1}{2|n|+d}\Big)= \int_{\H^d}  e^{i\frac{s}{2|n|+d}} \,\cW\Big(n,n,1,\frac{Y}{\sqrt{2|n|+d}}\Big)\, f(Y,s) \,dY\,ds \, . \eeq
Invoking  \eqref{definmeasure}, this gives rise to
  \begin{multline*}  
\langle d\sigma_{\S_{\wh\H^d}},\theta\rangle_{\cS'(\wh\H^d)\times\cS(\wh \H^d)} = 
\sum_{n\in\N^d} \frac 2{(2|n|+d)^{d+1}}\\
\times \int_{\H^d}   \cos\Big(\frac{s}{2|n|+d}\Big) \,\cW\Big(n,n,1,\frac{Y}{\sqrt{2|n|+d}}\Big)\, f(Y,s) \,dY\,ds\, ,
 \end {multline*}
which  by an obvious change of variable ends the proof of Formula\refeq{defG}.
\medskip

Furthermore, we have  the following classical combinatorial identity  
\beq
\label {combinatory} \# \big\{n\in\N^d, |n|= \ell \big\} = \begin{pmatrix} \ell +d-1 \\ \ell \end{pmatrix} \, ,\eeq  
and since the modulus of the Wigner transform of the (renormalized) Hermite functions defined by\refeq{defW} is bounded by one, Stirling's formula implies  that   $G $ belongs to~$ L^\infty(\H^d)$. 
The result follows. \end{proof} 
 \medbreak

\begin{remark}\label{sphereR}
{\sl Arguing as for the unit sphere, we readily gather that  for all $\theta$ belonging to $\cF_{\H}(\cS_{\rm {rad}}(\H^d))$, there holds
$$
\langle d\sigma_{\S_{\wh\H^d}({\sqrt R})},\theta\rangle_{\cS'(\wh\H^d)\times\cS(\wh \H^d)}=\langle G_R,{}^t\!\cF_{\H}\theta\rangle_{\cS'(\H^d)\times \cS(\H^d)}  \, ,
$$
where, in view of{\rm\refeq{newFourierconvoleq-1}}, $G_R$ is given by
\begin{equation}\label{compGR}
 G_R(Y,s)\eqdefa R^{d} (G  \circ \d_{{\sqrt R}})(Y,s)\,.
\end{equation}
We thus recover Formula $(19)$ derived by  M\"uller in\ccite{Muller}.} 
\end{remark}

\subsection {The restriction theorem} \label {restrictionmuller} 

Now let us state the restriction theorem, due to M\"uller in\ccite{Muller}, and sketch its proof for the convenience of the reader.
\begin{theorem}[\cite{Muller}]
\label {RestMuller}
{\sl 
If $1 \leq p \leq 2$, then 
\begin{equation}\label{restsphere} 
\| \cF_\H(f)_{|\S_{\wh\H^d}} \|_{L^{2}({\S_{\wh\H^d}})}\leq C_p \|f\|_{L^{p}_{Y}L^{1}_{s}}\, ,
\end{equation}  
for all radial functions $f$ in $\cS_{\rm {rad}}(\H^d)$.}
\end{theorem}

\medbreak
\begin{remark}\label{sphererk} {\sl In light of \eqref{definmeasure},  Theorem {\rm\ref{RestMuller}} writes 
$$
\begin{aligned}
\Big[\sum_{n\in\N^d} \frac1{(2|n|+d)^{d+1}}  \Big(\Big|\cF_\H(f)(n,n, \frac{1}{2|n|+d}\Big)\Big|^2 + \Big|\cF_\H(f)(n,n, \frac{-1}{2|n|+d}\Big)\Big|^2\Big)\Big]^{\frac 1 2} \\\leq C_p \|f\|_{L^{p}_{Y}L^{1}_{s}}\, ,
\end{aligned}
$$ 
for $f$ in $\cS_{\rm {rad}}(\H^d)$. Besides, since $\cS(\H^d)$ is dense in the space~$L^{p}_{Y}L^{1}_{s}(\H^d)$, when the estimate of Theorem~{\rm\ref{RestMuller} }holds we can  define~$\cF_\H f $ on $\S_{\wh\H^d}$ (a.e.\ with respect to $d\sigma_{\S_{\wh\H^d}}$), for each function~$f$ in $L^{p}_{Y}L^{1}_{s}$. 
Let us also emphasize that 
the gain we get here with respect to the horizontal variable $Y$ is better than the one obtained in Euclidean case, since the index~$p$ ranges from~$1$ to $2$ with no further restriction, contrary to the Euclidean case (recall Theorem~{\rm\ref{t:tsbase}}). Note that a counterexample for large values of~$p$ (namely~$p>4d/(2d-1)$, which is the dual index to $p_d$ introduced in Remark~{\rm\ref{rkindices}}) is provided in~{\rm\cite{Muller}}.} 
\end{remark}
\medbreak

\begin{remark}\label{dualestsphere} {\sl 
By duality, Inequality~{\rm\refeq{restsphere}} is equivalent to the following estimate
\begin{equation} 
\| \cF^{-1}_\H(\theta_{|\S_{\wh\H^d}}) \|_{L^{p'}_{Y}L^{\infty}_{s}}\leq C_p \, \|\theta_{|\S_{\wh\H^d}}\|_{L^{2}({\S_{\wh\H^d}})}\, ,
\end{equation}  
for all~$2 \leq p' \leq \infty$ and all   $\theta$ in $\cF_\H(\cS_{\rm {rad}})(\H^d)$.} 
\end{remark}

\begin{proof}[Proof of Theorem~{\rm\ref{RestMuller}}] First note that the case when $p=1$ is a straightforward consequence   of~\eqref{embtH}.  Then by   interpolation, it suffices to investigate the case when $p=2$.      
  For that purpose,    we shall proceed using  purely Fourier analysis arguments, and    sketch the  proof due to     M\"uller (see\ccite{Muller} for further details). 
Our goal here is to establish the following estimate  $$  
 \sum_{\ell\in\N}  \Big( \Big|\wt\Theta\Big(\ell, \ell, \frac{1}{2\ell+d}\Big)\Big|^2 + \Big|\wt\Theta\Big(\ell, \ell, \frac{-1}{2\ell+d}\Big)\Big|^2\Big) \\\leq C_2 \|f\|_{L^{2}_{Y}L^{1}_{s}}^2\, ,
$$ 
where
 \begin{equation}\label{deftildetheta} 
 \begin{aligned} 
\wt\Theta(\ell, \ell, \lam) &\eqdefa  \Big[ \begin{pmatrix} \ell +d-1 \\ \ell \end{pmatrix}^{-1} \frac1{(2\ell +d)^{d+1}} \Big]^{\frac 1 2}  \sumetage {n\in\N^d} {|n| = \ell} \cF_\H f (n, n, \lam) \\ &= \Big[ \begin{pmatrix} \ell +d-1 \\ \ell \end{pmatrix}^{-1} \frac1{(2\ell +d)^{d+1}} \Big]^{\frac 1 2} \int_{\H^d}  e^{-is\lam} \wt \cW( \ell, \lam,Y)\, f(Y,s) \,dY\,ds\,.
\end{aligned}
 \end{equation}
This amounts  to proving that  the operator $T$ defined by 
$$ 
Tf \eqdefa \Big(\wt\Theta\Big(\ell, \ell, \frac1{2\ell +d}\Big)\Big)_{\ell\in\N} \,,
$$ 
with~$\wt\Theta$ obtained from~$f$ through~(\ref{deftildetheta}),
 is bounded from $L^{2,1}(\H^{d})$ into $\ell^2(\N)$, or equivalently that its adjoint $T^*$  is bounded from~$\ell^2(\N)$ into $L^{2,\infty}(\H^{d})$. 

 \medskip 
 
 Now for any sequence~$\underline a = (a_\ell)_{\ell\in\N}$ in $\ell^2(\N)$, the operator $T^*$ is given  by  
 $$
 T^*(\underline a) (Y,s)= \sum^\infty_{\ell=0} a_\ell \, e^{-is\frac1{2\ell +d}} \, K\Big( \ell\virgp \frac1{2\ell +d} \virgp Y\Big) \, ,
 $$
 with 
 $$
  K\Big( \ell\virgp \frac1{2\ell +d} \virgp Y\Big)\eqdefa \Big[ \begin{pmatrix} \ell +d-1 \\ \ell \end{pmatrix}^{-1} \frac1{(2\ell +d)^{d+1}} \Big]^{\frac 1 2} \wt \cW(\ell, \lam,Y)\, .
 $$
 But by Lemma 4.2 in\ccite{Muller}, we know that for all $\ell$, $m$ in $\N$
\begin{equation} 
\label {orth} \int_{T^\star\R^d} \Big|K\Big( \ell\virgp \frac1{2\ell +d} \virgp Y\Big)\Big| \,  \Big|K\Big( m\virgp \frac1{2m +d} \virgp Y\Big)\Big| dY= \cO\Big( \frac 1 {\max(\ell, m)}\Big)\,.
\end{equation}  
 We deduce that 
 $$
 \|T^*(\underline a)\|_{L^{2,\infty}(\H^{d})} \lesssim \sum_{\ell\leq m} \frac { |a_\ell|  |a_m|}m =  \sum_{ m} |a_m| b_m\,,
 $$
 with $\ds b_m\eqdefa \frac {1 }m \sum_{\ell\leq m} |a_\ell| .$ 
  This ensures  the result  thanks to the following Hardy inequality (see\ccite{askey}):
   \begin{equation}\label{hardy}
   \|\underline b\|_{\ell^p(\N)}\leq C_p \|\underline a\|_{\ell^p(\N)}\, , 
   \end{equation}
   available for all $1< p < \infty$,      which achieves the proof of Theorem \ref{RestMuller}. \end{proof}
      \section{Proof of Strichartz estimates}\label{proofmainth}
  \subsection{Restriction theorem  on $\wh \R \times \wh \H^d$}\label {restSchr1}     
As we explained in Section~\ref{Erren} (cf.\ in particular Remark \ref{r:erren}),  to apply efficiently restriction estimates to PDEs to get Strichartz estimates, one has to investigate the Tomas-Stein restriction theorem on submanifolds  of  the product~$\wh\R \times \wh \H^d$, where $\wh \R$ stands for  the dual group of  $\R$. In the following we set
    $$
    {\mathbb D}\eqdefa \R \times  \H^d \quad \mbox{and}\quad \wh    {\mathbb D}\eqdefa \wh\R \times \wh \H^d\,.
    $$ 

         \subsubsection{The Fourier transform on $\R \times  \H^d$} 
    A combination of Fourier analysis on the real line and on the Heisenberg group leads directly to an efficient Fourier theory on ${\mathbb D}$ whose Haar measure  is obviously nothing else than  the Lebesgue measure.  We define the Fourier transform of $f$ in $L^1( {\mathbb D})$  as follows:
 \beq
\label {defFD}
\cF_{\mathbb D} f(\alpha, \wh w) \eqdefa \int_{{\mathbb D}}  \overline{e^{it\alpha} e^{is\lam}  \cW(\wh w,Y)}\, f(t,Y,s) \, dt \,dY\,ds\,,
 \eeq for any $(\alpha, \wh w) \in \wh    {\mathbb D}$.
  The Fourier transform $\cF_{\mathbb D}$  inherits all the properties of   $\cF_{\H}$ and of the Fourier transform on the real line $\cF$. In particular,  the inversion and Fourier-Plancherel formulae   take the following forms:
\beq
\label {inverseFourierD}f(t, w) = \frac {2^{d-2}}  {\pi^{d+2} }   \int_{\wh    {\mathbb D}} 
e^{it\alpha}  e^{is\lam} \cW(\wh w, Y)\cF_{{\mathbb D}} f(\alpha, \wh w) \, d\alpha  \, d\wh w 
 \eeq
and
 \beq
\label {FPD}
(\cF_{{\mathbb D}} f|\cF_{{\mathbb D}} g)_{L^2(\wh    {\mathbb D})}  = \frac {\pi^{d+2}} {2^{d-2}} (f|g)_{L^2({\mathbb D})}\, .
\eeq
 In the sequel, we shall say that a function $f$ on  $ {\mathbb D}$ is  radial if it is invariant under the action of~$U(d)$, in the sense that for any $R$ of~$U(d)$ and any $(t, Y, s)$ of $ {\mathbb D}$, we have
 $$ f(t, R(Y),s)= f(t, Y,s) \, .$$  It readily stems from Relation\refeq{newFourierrad}  that if $f$ belongs to~$L_{\rm {rad}}^1( {\mathbb D})$, then for all $(\alpha, \wh w) \in \wh    {\mathbb D}$, 
 \beq
\label {newFourierradD} \cF_{{\mathbb D}} (f) (\alpha, n, m, \lam)= \cF_{{\mathbb D}} (f)  (\alpha, |n|, |n|, \lam) \delta_{n,m} \,.
\eeq
To avoid any confusion, we shall denote in what follows by $\star_{\mathbb D}$ the noncommutative convolution product on~${\mathbb D}$, namely
\beq
\label {definConvolD}
f \star_{\mathbb D} g ( t, w ) \eqdefa \int_{{\mathbb D}} f (t-t',  w \cdot v^{-1} ) g(t',  v)\,dt' \, dv 
\, ,
\eeq
which of course enjoys Young's inequalities and satisfies
 \beq
\label {newFourierconvoleqD}
 \cF_{{\mathbb D}} (f\star_{\mathbb D} g) (\alpha, \wh w)   =\sum_{p\in \N^{d}} \cF_{{\mathbb D}} f(\alpha, n,p,\lam)\cF_{{\mathbb D}} g(\alpha, p,m,\lam)\,.
 \eeq 
 Patching    Fourier analysis on the real line and on  the Heisenberg group,   one can easily check that~$L_{\rm {rad}}^1( {\mathbb D})$ is a  commutative sub-algebra of~$L^1( {\mathbb D})$ where\refeq{newFourierconvoleqD} reduces to
 \beq
\label {newFourierconvoleqDrad}
 \cF_{{\mathbb D}} (f\star_{\mathbb D} g) (\alpha, |n|, |n|, \lam)    =\cF_{{\mathbb D}} f(\alpha, |n|, |n|, \lam)\cF_{{\mathbb D}} g(\alpha, |n|, |n|, \lam)\,.
 \eeq 
 Besides, it is worth noticing that $\cF_{{\mathbb D}}$ is a bicontinuous isomorphism between the space~$\cS({\mathbb D})$ (which coincides with~$\cS(\R^{2d+2})$) and~$\cS(\wh {\mathbb D})$ --- which can be defined naturally from the definition of~$\cS({\wh \H^d})$.
 The map~$\cF_{{\mathbb D}}$ can  then be continuously extended from~$\cS'({\mathbb D})$ into~$\cS'(\wh    {\mathbb D})$ by duality 
according to the following formula:
$$
\cF_{{\mathbb D}}: \left \{
\begin{array}{ccl}
\cS'({{\mathbb D}}) & \longrightarrow & \cS'(\wh {\mathbb D}) \\
T & \longmapsto &  \Bigl[\theta \mapsto  \langle T,{}^t\!\cF_{\mathbb D}\theta\rangle_{\cS'({\mathbb D})\times \cS({\mathbb D})}\Bigr]\, ,
\end{array}
\right.
$$
with
 \beq
\label {linetFF-1D}
 {}^t\cF_{\mathbb D}\theta(t, y,\eta,s) \eqdefa \frac {\pi^{d+2}} {2^{d-2}} (\cF_{\mathbb D}^{-1}\theta)(-t, y,-\eta,-s) \, .
\eeq
\subsubsection{A surface measure}\label{section-surface measure}
Let us define the set     \begin{equation}
\label {eqparaboloid} 
\Sigma\eqdefa \Big\{(\alpha,\wh w)=\big(\alpha, (n, n, \lam) \big) \in \wh {\mathbb D}\, /\, \alpha= 4  |\lam|(2|n|+d) \Big\} \, .
\end{equation}
We endow $\Sigma$ with the measure $d\Sigma$ induced by the projection $\pi: \wh \R \times \wh \H^{d}\to \wh  \H^{d}$ onto the second factor. Following the notation of Section~\ref{Erren}, $d\Sigma=(\pi|_{\Sigma})^{-1}_{\sharp} d\wh w$.  More explicitly, recalling \eqref{dublev}, for all $\Theta$  in $\cS(\wh {\mathbb D})$, we have 
 $$
 \begin{aligned} 
\langle d\Sigma,\Theta\rangle_{\cS'(\wh {\mathbb D})\times\cS(\wh{\mathbb D})}& = \int_{\Sigma }\Theta(\alpha,\wh w) \, d\Sigma(\alpha,\wh w) \\
&\eqdefa\sum_{n\in\N^d}  {c_n^{d+1}} \int^\infty_0 \Big(\Theta\big(\alpha,n,n, {\alpha}{c_n}\big)
+ \Theta\big(\alpha,n,n, {-\alpha}{c_n}\big)
\Big) \, \alpha^d \, d\alpha\, ,
\end {aligned}
$$
where to simplify notation we have set
$$
c_n\eqdefa \big(4(2|n|+d)\big)^{-1}\,.
$$
Notice that if $\Theta:\Sigma\subset \wh {\mathbb D} \to \C$ is defined as $\Theta=\theta\circ \pi|_{\Sigma}$, where $\theta:\wh \H^{d}\to \C$, then by construction for all $1\leq p\leq \infty$
\begin{equation}\label{lpislp}
\|\Theta\|_{L^{p}(\Sigma,d\Sigma)}=\|\theta\|_{L^{p}(\wh \H^{d})}\, .
\end{equation}

Our purpose here is to show that every (appropriate) function $f$ has a Fourier transform~$\cF_{\mathbb D}f$ that restricts to $\Sigma$.  Actually as in the classical case, this restriction property is best dealt with in compact subsets of $\Sigma$. Thus, we shall consider $\Sigma$   endowed with the surface measure $d\Sigma_{{\rm loc}}\eqdefa\psi(\alpha)d\Sigma$ defined by 
\begin{equation} \label {measureparaboloid}
 \begin{aligned} 
  \int_{\Sigma} \Theta(\alpha,\wh w) \,   d\Sigma_{{\rm loc}}(\alpha,\wh w) &\eqdefa \sum_{n\in\N^d}  {c_n^{d+1}}\\
&     \times \int^\infty_0 \Big(\Theta\big(\alpha,n,n, {\alpha }{c_n}\big)
+ \Theta\big(\alpha,n,n, {-\alpha }{c_n}\big)
\Big) \, \alpha^d \, \psi(\alpha) \, d\alpha\, , \end {aligned}
\end{equation}
with $\psi$   any    smooth, nonnegative, even  function,  compactly supported in~$\R$ with an $L^\infty$ norm  at most $1$.

\medskip Proceeding as for  the restriction theorem  on   the sphere     of  $\wh\H^d $, let us first compute
$$
G_{\Sigma_{{\rm loc}}}
\eqdefa  \cF_{{\mathbb D}}^{-1}(d\Sigma_{{\rm loc}})\,.
$$

\begin{proposition}
\label {deffouriereSigma}
{\sl With the above notation, $G_{\Sigma_{{\rm loc}}}$ is the bounded function on~${\mathbb D}$ defined by  
\beq
\label {defGSigma} G_{\Sigma_{{\rm loc}}}(t, w)=  2\pi \, \int^\infty_0 G_{\alpha }(w) \, e^{-i t\,\alpha} \psi(\alpha)\,  d\alpha  \,, 
\eeq where $G_R$
is given  by \eqref{compGR}.}\end{proposition} \begin{proof} 
Arguing as in the proof of   Proposition \ref{deffouriereS} and using the fact that the Fourier transform~$  \cF_{\mathbb D}$ is a bicontinuous isomorphism between the spaces~$\cS({\mathbb D})$ and~$\cS(\wh{\mathbb D}),$
there holds (with $\Theta \eqdefa  \cF_{\mathbb D} f $)  
\beq
\label {defFouriermeasparaboloid} \langle d{\Sigma_{{\rm loc}}},\Theta\rangle_{\cS'(\wh {\mathbb D})\times\cS(\wh {\mathbb D})} = \Big(\frac \pi  2 \Big)^d    \langle G_{\Sigma_{{\rm loc}}} , \wt f\rangle_{\cS'({\mathbb D})\times\cS({\mathbb D})}  \, ,
\eeq
with $\wt f (t, y,\eta,s) \eqdefa f(-t, y,-\eta,-s)$.
 By definition,  we have for any non negative real number $\alpha$
$$
\, \, \Theta\big(\alpha,n,n, {\pm\alpha }{c_n}\big)= \int_{{\mathbb D}}  e^{\mp i s{\alpha  }{c_n}} \,\cW\Big(n,n,1, { {\alpha^\frac{1}{2} {\sqrt{c_n}}\,} Y}\Big)\, e^{-i t\,\alpha}f(t,Y,s) \,dt \,dY\,ds \, ,
$$
which implies in view of  \eqref{measureparaboloid}  that
 \begin{multline*}  
\langle d{\Sigma_{{\rm loc}}},\Theta\rangle_{\cS'(\wh {\mathbb D})\times\cS(\wh {\mathbb D})} = 
\sum_{n\in\N^d}   2{c_n^{d+1}}\\
\times \int^\infty_0 \int_{{\mathbb D}}  \cos\Big( s{\alpha   }{c_n}\Big) \,\cW\Big(n,n,1, { {\alpha^\frac{1}{2} } {\sqrt{c_n}}\,Y}\Big)\, e^{-i t\,\alpha}f(t,Y,s) \,dt \,dY\,ds \, \alpha^d \psi(\alpha) \, d\alpha   \, .
 \end {multline*}
This achieves the proof of the result thanks to the Fubini theorem and  Formula \eqref{compGR}, after an obvious change of variables.
\end{proof} 

\subsubsection{Restriction theorem} 
 Our aim now   is to  establish the following   restriction  result for the ``dual set"~${\Sigma_{{\rm loc}}}$   of $\wh {\mathbb D}$ defined by~\eqref{eqparaboloid} and endowed with the measure~$ d{\Sigma_{{\rm loc}}}$. \begin{theorem}
\label {Restpraphase}
{\sl 
If $1 \leq q \leq  p \leq 2$, then 
\begin{equation}\label{restSig} 
\|   \cF_{\DD}(f)_{|{\Sigma_{{\rm loc}}}} \|_{L^{2}( d{\Sigma_{{\rm loc}}})}\leq C_{p,q} 
\|f\|_{L^1_s L_t^{q}L^{p}_{Y}}\, ,
\end{equation}  
for all radial functions $f$ in $\cS_{\rm {rad}}(\R \times \H^d)$. }
\end{theorem}

\medbreak
\begin{remark}\label{dualestpara} {\sl 
By duality, Theorem~{\rm\ref{Restpraphase}} may be rephrased as follows : for any~$2 \leq p' \leq q' \leq \infty$, there holds 
   \begin{equation}\label{restSigduale} 
    \| \cF_{\DD}^{-1}( \Theta_{|\Sigma_{{\rm loc}}} )\|_{L^\infty_s L_t^{q'} L^{p'}_{Y}} \leq C_{p,q} \|   \Theta_{|{\Sigma_{{\rm loc}}}} \|_{L^{2}( d{\Sigma_{{\rm loc}}})} \, 
    ,\end{equation} for all $\Theta \in \cF_{\DD}(\cS_{\rm {rad}}(\R \times \H^d))$.} 
\end{remark}

\begin{proof}[Proof of Theorem~{\rm\ref{Restpraphase}}]  We   handle differently the cases  $1 \leq p < 2$ and~$p=2$. To undertake the case $1 \leq p < 2$,  we shall follow the Euclidean strategy outlined in Section\refer{Erren}.  To this end, let us  introduce~$R_{\Sigma_{{\rm loc}}}$  the restriction operator on~$\Sigma_{\rm loc}$ defined for  any function~$f$ in~$\cS({\mathbb D})$   by
$$
R_{\Sigma_{{\rm loc}}} f \eqdefa\cF_\DD(f)_{|{\Sigma_{{\rm loc}}}} = \Big(\int_{\DD}  \overline{e^{is\lam+it\alpha} \cW(\wh w,Y)}\, f(t,Y,s) \,dY\,ds\,dt\Big)_{|{\Sigma_{{\rm loc}}}}\,,
$$ 
and by $R_{\Sigma_{{\rm loc}}}^*$ its adjoint.  By definition 
 $$
 \begin{aligned}
 \langle R_{\Sigma_{{\rm loc}}} f,R_{\Sigma_{{\rm loc}}} f\rangle_{L^2 (\Sigma_{\rm loc})}&= \sum_{n\in\N^d}  {c_n^{d+1}}\int_0^\infty \Big(\big|\Theta\big(\alpha,n,n, {\alpha }{c_n}\big)\big|^2
\\
&\qquad\qquad +\big|\Theta\big(\alpha,n,n, {-\alpha }{c_n}\big)\big|^2\Big)\psi(\alpha)\alpha^d\,d\alpha\, ,
\end{aligned}$$ 
with $\ds \Theta \eqdefa\cF_\DD (f )$, which  readily implies that
 \begin{multline}
   \label{restradial} 
\langle R_{\Sigma_{{\rm loc}}} f,R_{\Sigma_{{\rm loc}}} f\rangle_{L^2 (\Sigma_{{\rm loc}})} = \sum_{n\in\N^d}   {c_n^{d+1}} 
 \\ 
\times\int_0^\infty \Big[\Theta\big(\alpha,n,n, {\alpha }{c_n}\big)\int_{\DD}  e^{ {is}{\alpha  c_n}+it\alpha} \cW\Big(n,n, {\alpha }{c_n },Y\Big)\, \overline{f(t,Y,s)} \,dt\, dY\,ds\\
 \qquad  \qquad+ \,\Theta\big(\alpha,n,n, {-\alpha }{c_n}\big)\int_{\DD}  e^{ {-is}{\alpha  c_n}+it\alpha} \cW\big(n,n, {-\alpha }{c_n},Y\big)\, \overline{f(t,Y,s)} \,dt\, dY\,ds\Big] \psi(\alpha)\alpha^d\,d\alpha\, .
\end {multline}  
In view of \eqref{measureparaboloid}, this leads  to 
\begin{equation} \label{defRR*Hsigma}
\langle R_{\Sigma_{{\rm loc}}} f,R_{\Sigma_{{\rm loc}}} f\rangle_{L^2 (\Sigma_{{\rm loc}})}=\langle R_{\Sigma_{{\rm loc}}}^*R_{\Sigma_{{\rm loc}}} f, f\rangle_{L^2 (\DD)}\,, \end{equation}  
with (for $w=(Y,s)$)
 \begin{equation}  \label{restradial} 
(R_{\Sigma_{{\rm loc}}}^*R_{\Sigma_{{\rm loc}}} f)(t,w)=  \int_{\Sigma } e^{is\lam+it\alpha} \cW(\wh w,Y) \cF_\DD(f) (\alpha,\wh w) d{\Sigma_{{\rm loc}}} (\alpha,\wh w)\, .
 \end{equation}
Since $f$   belongs to  $\cS_{\rm {rad}}(\DD)$,  combining \eqref {newFourierrad} together with~\eqref{princident} and \eqref{eq:fouriertranslate} we infer that 
 the operator~$R_{\Sigma_{{\rm loc}}}^*R_{\Sigma_{{\rm loc}}}$ writes  in the  radial setting\footnote{where of course
 $ (f \circ \tau_{(\tau, w)})(t, v)= f(t+\tau, w \cdot v)$.} :
 \begin{equation} \label{defRR*Hrad}
 (R_{\Sigma_{{\rm loc}}}^*R_{\Sigma_{{\rm loc}}} f)(t,Y,s)=  \int_{\Sigma } \cF_{\DD}(f \circ \tau_{(-t, Y, -s)})(\alpha,\wh w)\,  \, d{\Sigma_{{\rm loc}}}(\alpha,\wh w)\, .
  \end{equation} 
By~\eqref{eq:defconv} and~\eqref{defFouriermeasparaboloid},  this gives rise to  
\begin{equation} 
\label{defRR*Hradconv}
(R_{\Sigma_{{\rm loc}}}^*R_{\Sigma_{{\rm loc}}} f)(t, Y,s)  = \Big(\frac \pi  2 \Big)^d   \, (G_{\Sigma_{{\rm loc}}} \star_{\DD} \check  f)(-t,-Y,s)\, ,
\end{equation} 
for all $f$   in $\cS_{\rm {rad}}(\DD)$. 

 \medskip Now applying the H\"older inequality to \eqref{defRR*Hsigma}, we    deduce that    $$
   \begin{aligned}
   \| R_{\Sigma_{{\rm loc}}} f\|^2_{L^{2}({{\Sigma_{{\rm loc}}}})}&\leq \| R_{\Sigma_{{\rm loc}}}^* R_{\Sigma_{{\rm loc}}} f\|_{L^\infty_s L_t^{q'} L^{p'}_{Y}}\|f\|_{L^1_s L_t^{q}L^{p}_{Y}}\\
    & \leq C_d\| \check f \star_{\DD} G_{\Sigma_{{\rm loc}}} \|_{L^\infty_s L_t^{q'} L^{p'}_{Y}}\|f\|_{L^1_s L_t^{q}L^{p}_{Y}}\,,
\end{aligned}
   $$
     for some irrelevant constant~$C_d$ which may change from line to line.
   Then   as in the Euclidean case, to complete the proof of  Estimate \eqref{restSig},  we are reduced to proving that~$R_{\Sigma_{{\rm loc}}}^* R_{\Sigma_{{\rm loc}}}$ is bounded from $L^1_s(\R, {L_t^{q}(\R,L^{p}( T^*\R^{d})))}$ into $L^\infty_s(\R, {L_t^{q'}(\R,L^{p'}( T^*\R^{d})))}$. For that purpose,  let us start by observing that in light of\refeq{compGR} and\refeq{defGSigma}, we have 
 \begin{multline*}
 f\star_{\DD} G_{\Sigma_{{\rm loc}}}  (t,Y,s)=C_d \sum_\ell \frac 1 {(2\ell+d)^{d+1} } \\ \times \sum_{\pm}\int_0^\infty
 f\star_{\DD} \Big(
 e^{\pm i \alpha_{\ell}  - it\alpha} \widetilde \cW \Big(\ell,  \alpha_{\ell}  , Y\Big) 
 \Big) \alpha^d \, \psi(\alpha) d\alpha \,,
 \end{multline*} 
 where $\wt \cW$ is given by\refeq{princident} and where we have defined
 $$
 \alpha_{\ell} \eqdefa\frac{\alpha }{4(2\ell+d)}
 \,\cdotp
 $$
An easy computation shows that for any real number~$\lambda$,
 $$
 f\star_{\DD} \Big(
 e^{\pm i\lambda s - it\alpha} \widetilde \cW (\ell,\lambda,Y)
 \Big) = e^{\pm i\lambda s - it\alpha} \left(f^{\alpha,\pm \lambda}
 \star_{\pm\lambda} \widetilde \cW (\ell,\lambda,Y)
\right) \, , $$
where~$ \star_{ \lambda}$  denotes the twisted convolution defined in \eqref{twisted}, 
and with
\begin{equation}\label{defnot}
f^{\alpha,\beta}(Y)
\eqdefa \int  e^{ i s\beta + it\alpha} f(t, Y,s) dt ds = \int  e^{ i s\beta} \cF_\R f(-\alpha, \cdot)  ds \, .
\end{equation}
Defining~$\psi_+(\alpha) \eqdefa \psi (\alpha) {\mathbf 1}_{\alpha>0}$, it  follows that
$$
 \begin{aligned}
 f\star_{\DD} G_{\Sigma_{{\rm loc}}}  (t,s,Y)&= C_d \sum_{\ell}   \frac 1 {(2\ell+d)^{d+1} } \\ 
 &\qquad \times\sum_{\pm}\int_0^\infty
  e^{\pm is  \alpha_{\ell} -it\alpha }
\Big(f^{\alpha,\pm \alpha_{\ell} }\star_{\pm  \alpha_{\ell} }  \widetilde \cW \big(\ell, \alpha_{\ell} ,Y\big)\Big)
\alpha^d \,  \psi(\alpha) \,d\alpha \\
& = C_d \sum_{\ell} \sum_{\pm}\frac{1}{(2\ell+d)^{d+1}}
\cF_{\alpha\to t} \Big(e^{\pm is \alpha_{\ell}  }\alpha^d\psi_+(\alpha)
 \Big(f^{\alpha,\pm \alpha_{\ell} }\star_{\pm \alpha_{\ell} } \widetilde \cW  \big(\ell,  \alpha_{\ell} ,Y \big)\Big)\Big)\, .
 \end{aligned} 
$$
 Now let us fix~$s\in \R$ and~$Y\in T^*\R^{d}$, and  compute the~$L_t^{q'}(\R)$ norm of the function of~$t$  appearing on the right-hand side of the equality,
   for~$q' \geq 2$: by the Hausdorff-Young inequality we find  
  $$
      \| f\star_{\DD} G _{\Sigma_{{\rm loc}}}   \|_{L_t^{q'}}  \lesssim  \sum_{\ell} \sum_{\pm}\frac{1}{(2\ell+d)^{d+1}} \Big  \|e^{\pm is \alpha_{\ell}  }\alpha^d\psi(\alpha)
 \Big(f^{\alpha,\pm \alpha_{\ell} }\star_{\pm \alpha_{\ell} } \widetilde \cW  \big(\ell, \alpha_{\ell} ,Y \big)\Big)\Big\|_{ L_\alpha^q} 
 \, .
$$
Now noticing that as soon as~$q' \geq p'\geq2$, thanks to  Minkowski's inequality  we have  for any function~$g$ defined on~$ \R \times T^*\R^{d}$
 \begin{equation} \label{intcomp}
  \begin{aligned}
\|\cF_{\alpha\to t} g\|_{L^{q'}_{t} L^{p'}_{Y}} &\leq \|\cF_{\alpha\to t} g\|_{L^{p'}_{Y} L^{q'}_{t}  }\\
&\lesssim  \| g\|_{ L^{p'}_{Y} L^{q}_{\alpha}}\\
&\lesssim  \| g\|_{L^{q}_{\alpha}  L^{p'}_{Y}  } \,, 
 \end{aligned} 
 \end{equation}
we deduce  that for~$q' \geq p'>2$ (since we assumed here that $1 \leq p < 2$)
$$
  \begin{aligned}
    \| f\star_{\DD} G_{\Sigma_{{\rm loc}}}   \|_{L^\infty_s L_t^{q'} L^{p'}_{Y} }  &\lesssim  \sum_{\ell} \sum_{\pm}\frac{1}{(2\ell+d)^{d+1}}\Big  \|\alpha^d\psi(\alpha)
 \Big(f^{\alpha,\pm \alpha_{\ell} }\star_{\pm  \alpha_{\ell} } \widetilde \cW  \big(\ell, \alpha_{\ell} , Y \big)\Big)\Big\|_{  L_\alpha^q L^{p'}_{Y} } 
 \, .
 \end{aligned} 
 $$
  But by\refeq{est2}, there holds
$$
\Big  \| 
\Big(f^{\alpha,\pm \frac\alpha{4(2\ell+d)}}\star_{\pm \frac\alpha{4(2\ell+d)}}  \widetilde \cW \Big(\ell, \frac{\alpha }{4(2\ell+d)},Y\Big)\Big)
 \Big\|_{L^{p'}_{Y}} \lesssim \ell^{d-1+ \frac {2} {p'}}\,\alpha^{-  \frac {2d} {p'}} \| \cF_\R f(-\alpha, \cdot)\|_{L^{p}_{Y}L^{1}_{s}}\,,
 $$ 
which  (since $p'>2$)  implies that 
$$\| f\star_{\DD} G_{\Sigma_{{\rm loc}}}   \|_{L^\infty_s L_t^{q'} L^{p'}_{Y}}  \lesssim \Big \| \|  \cF_{\R}(f)(-\alpha, \cdot)\|_{L^{p}_{Y}L^{1}_{s}} \,
 \alpha^{d(1- \frac {2} {p'})}\ \psi(\alpha)\Big\|_{L_\alpha^{q}}   \,.$$
Then, applying a H\"older estimate in~$\alpha$ followed by the Hausdorff-Young inequality, we get for any~$a \geq 2$ 
$$
 \begin{aligned} \| f\star_{\DD} G_{\Sigma_{{\rm loc}}}   \|_{L^\infty_s L_t^{q'} L^{p'}_{Y}} &\lesssim  \|  \cF_{\R}(f)\|_{L^{a}_\alpha  L^{p}_{Y}L^{1}_{s}} \|  \alpha^{d(1- \frac {2 } {p'})}\ \psi(\alpha)\|_{L_\alpha^{b}}  \\
  & \lesssim\|f\|_{L^{a'}_t L^{p}_{Y}L^{1}_{s}} \|  \alpha^{d(1- \frac {2 } {p'})}\ \psi(\alpha)\|_{L_\alpha^{b}(\R)} \, ,  \end{aligned}
 $$
where of course $ a'$  is the conjugate exponent of $a$ and  $\ds \frac 1 a+ \frac 1 b = \frac 1 q \cdot$ 

\medskip
Finally selecting $a'=q$ and thanks again to Minkowski's inequality, we get for all real numbers~$q' \geq p'>2$  
$$
 \| f\star_{\DD} G_{\Sigma_{{\rm loc}}}   \|_{L^\infty_s L_t^{q'} L^{p'}_{Y}}  \lesssim\|f\|_{L^1_s L_t^{q} L^{p}_{Y} }
 $$ 
 which   completes the proof of the result in the case when $1 \leq q \leq  p < 2$.

 \bigskip
      Finally the proof in the case when $p=2$ is in the same spirit than the one concerning the unit dual sphere of the Heisenberg group already outlined in the proof of Theorem \ref{RestMuller}. 
By definition, our aim here is to show that (for $q \leq 2$)
$$  
\sum_{\ell\in\N}  \int_0^\infty  \Big( \Big|\wt\Theta\Big(\alpha, \ell, \ell, \frac{\alpha }{2\ell+d}\Big)\Big|^2 + \Big|\wt\Theta\Big(\alpha, \ell, \ell, \frac{-\alpha }{2\ell+d}\Big)\Big|^2\Big)\psi(\alpha)\alpha^d\,d\alpha \\\lesssim \|f\|^2_{L^1_s L_t^{q} L^{2}_{Y}}\, ,
$$ where  with the notation introduced in~(\ref{deftildetheta}) and on page \pageref{orth}
 \begin{equation} \label{413}
 \begin{aligned} 
\wt\Theta\Big(\alpha, \ell, \ell, \lam\Big) & \eqdefa \Big[ \begin{pmatrix} \ell +d-1 \\ \ell \end{pmatrix}^{-1} \frac1{(2\ell +d)^{d+1}} \Big]^{\frac 1 2}  \sumetage {n\in\N^d} {|n| = \ell} \cF_{\mathbb D} f (\alpha, n, n, \lam) \\ 
&= \int_{{\mathbb D}} e^{-it\alpha} e^{-is\lam} K(\ell, \lam,Y)\, f(t,Y,s) dt \,dY\,ds\,.
\end{aligned}
\end{equation} 
 For that purpose, let us establish that  the operator $T_{\mathbb D}$ defined on $\cS_{\rm {rad}}(\DD)$ by 
$$ 
T_{\mathbb D}f \eqdefa \Big(\wt\Theta\big(\alpha, \ell, \ell,  \alpha_{\ell} \big)\Big)_{\ell\in\N} \,,
$$  
where~$f$ is related to~$\Theta$ through~(\ref{413}), is bounded from $L^1_s(\R, {L_t^{q}(\R,L^{2}_{Y}( T^*\R^{d})))}$ into the space $L^2(\N \times \R)$ endowed with the measure $\ell^2(\N)\otimes L^2(\R^+, \psi(\alpha)\alpha^d\,d\alpha)$, or equivalently that its adjoint~$T_{\mathbb D}^*$  is bounded from~$L^2(\N \times \R)$ into $L^\infty_s(\R, {L_t^{q'}(\R,L^{2}_{Y}( T^*\R^{d})))}$. 

 \medskip  For $\underline a (\alpha)= (a_\ell (\alpha))_{\ell\in\N}$ in ~$L^2(\N \times \R)$, the operator $T_{\mathbb D}^*$ is given  by 
 $$T_{\mathbb D}^*(\underline a) (t, Y,s)= \sum^\infty_{\ell=0} \int_0^\infty a_\ell (\alpha)\, e^{-it\alpha} \, e^{-is \alpha_{\ell} } \, K\big(\ell, \alpha_{\ell} ,Y\big) \psi(\alpha)\alpha^d\,d\alpha\, .
 $$
Combining~\eqref{intcomp} together with the Hausdorff-Young inequality, we find that for any  fixed~$s$ in  $\R$,  there holds
 $$\|T_{\mathbb D}^*(\underline a) (\cdot, \cdot, s)\|_{ L_t^{q'}L^{2}_{Y}} \lesssim  \| g\|_{L^{q}_{\alpha}  L^{2}_{Y} } \,,$$
 where 
 $$ g(\alpha, Y)\eqdefa \alpha^d \psi(\alpha) \,\sum^\infty_{\ell=0} a_\ell (\alpha)\, K\big(\ell,  \alpha_{\ell} ,Y\big)\, . $$
 Now taking advantage of   \eqref{orth} and performing an obvious change of variable,  we get  
 $$
\Big \| \sum^\infty_{\ell=0} a_\ell (\alpha)\, K\big(\ell, \ell, \alpha_{\ell} ,Y\big)\Big\|^2_{L^{2}_{Y}} \lesssim \alpha^{d }    \sum_{ m}   |a_m(\alpha)|  b_m (\alpha)  \,,
 $$
 where of course  $\ds b_m(\alpha)\eqdefa\frac {1 }m \sum_{\ell\leq m} |a_\ell(\alpha)| \,.$  
 Thanks to  the Hardy inequality \eqref{hardy}, this ensures that 
 $$
  \| g(\alpha, \cdot)\|_{L^{2}_{Y} } \lesssim \alpha^{\frac  {3d  } 2 } \psi(\alpha) \| \underline a (\alpha)\|_{\ell^2(\N) }\, ,
  $$
 which by H\"older's inequality gives rise to
$$
\| g\|_{L^{q}_{\alpha}L^{2}_{Y} } \lesssim  \|\alpha^{d } \psi^{\frac  {1 } 2 }(\alpha) \|_{L^{\frac  {2-q } {2 q } }( \R) }\| \underline a \|_{L^2(\N \times \R) } \, .
 $$  
 This achieves the proof of the Fourier restriction  estimate\refeq{restSig}.    
\end{proof}
 
\begin{remark}\label{remarkwaves} {\sl
 In the case of the wave equation,   we consider the sets     \begin{equation}
\label {eqparaboloidw} 
\Sigma_{\pm} \eqdefa \Big\{(\alpha,\wh w)=\big(\alpha, (n, n, \lam) \big) \in \wh {\mathbb D}\, /\, \alpha^{2}= 4  |\lam|(2|n|+d)\, ,  \, \, \pm \alpha>0 \Big\} \, .
\end{equation}
Each of those is endowed with the measure induced by the projection $\pi: \wh \R \times \wh \H^{d}\to \wh  \H^{d}$ onto the second factor. Explicitly, for all $\Theta$  in $\cS(\wh {\mathbb D})$, we have 
 $$
 \begin{aligned} 
\langle d\Sigma_{\pm},\Theta\rangle_{\cS'(\wh {\mathbb D})\times\cS(\wh{\mathbb D})}& = \int_{\Sigma }\Theta(\alpha,\wh w) \, d\Sigma(\alpha,\wh w) \\
&\eqdefa\sum_{n\in\N^d}  {c_n^{d+1}} \int^\infty_0 \Big(\Theta\big(\alpha,n,n, {\alpha^{2}}{c_n}\big)
+ \Theta\big(\alpha,n,n, {-\alpha^{2}}{c_n}\big)
\Big) \, \alpha^d \, d\alpha\, .
\end {aligned}
$$
Following the same argument as above, one proves \eqref{restSigduale} for the corresponding localized measures.
}
\end{remark}

 \subsection{Strichartz estimates for the Schr\"odinger operator on $\H^{d}$ }\label {stSchr}

 The aim of this section is to prove the Strichartz estimate stated in Theorem~\ref{STth}.  {By duality arguments, one can reduce to the free Schr\"odinger equation (see Appendix\refer{proofprop} for further details about the inhomogeneous framework)}. Let~$u_0$ be a function in~$\cS_{\rm {rad}}(\H^d)$ and consider  the Cauchy problem
$$
(S_\H)\qquad \left\{
\begin{array}{c}
i\partial_t u -\D_\H u = 0\\
u_{|t=0} = u_0 \, .
\end{array}
\right. 
$$
As in the Euclidean case, Fourier analysis allows us to explicitly solve $(S_\H)$. More precisely, taking the partial Fourier transform with respect to the variable $w$ in the   Heisenberg group, we obtain    for all~$(t, \wh w)$ in~$\R \times \wt \H^d$
$$
\left\{
\begin{array}{rcl}
\ds  i \frac d   {dt}  \cF_\H(u) (t, n,m, \lam)&\! =\!  &- 4|\lam|(2|m|+d) \cF_\H(u) (t, n,m, \lam)\\
 \cF_\H(u)_{|t=0} &\!=\! &   \cF_\H  {u_0}\, .
\end{array}
\right.
$$
By integration, this  leads to
 $$
\cF_\H(u) (t, n,m, \lam)=   e^{4i t |\lam|(2|m|+d)}\cFH (u_0) (|n|, |n|, \lam) \delta_{n,m}\, .
$$
Then applying the  inverse Fourier  formula \eqref{inverseFourierH}, we infer that  the solution of the Cauchy problem $(S_\H)$ can be expressed as follows:
\beq
\label {solSH}u(t, Y, s)= \frac {2^{d-1}}  {\pi^{d+1} }   \int_{\wh \H^d} 
e^{is\lam} \, \cW(\wh w, Y) \, e^{4i t |\lam|(2|m|+d)} \, \cFH (u_0) (|n|, |n|, \lam) \delta_{n,m} \, d\wh w  \, . \eeq
That explicit representation of the solution to $(S_\H)$ can be re-expressed as the inverse Fourier transform in $\wh {\mathbb D}$ of $\cFH (u_0) \, d\Sigma$, where as defined in Paragraph~\ref{section-surface measure}, $d\Sigma$ is the measure in~$\wh {\mathbb D}$ given for any regular function $\Phi $ on $\wh {\mathbb D}$:
\beq
\label {formalmeasure} \int_{\wh {\mathbb D}} \Phi(\alpha,\wh w) \, d\Sigma(\alpha,\wh w)= \int_{\wh\H^d} \Phi(4|\lam|(2|m|+d),\wh w) \,d\wh w \, ,\eeq
and which is determined by Formula \eqref{measureparaboloid}.

\medskip  Now in order to establish Estimate\refeq{dispSTH}, let us  first discuss the case of  initial data frequency localized,  in the sense of Definition\refer{freqlocgen},
 in  the unit ball~${\mathcal B}_1$; then  we shall generalize this to any ball of radius~$\Lambda$ by a scaling argument, and finally  the result will follow by density.

\medskip So let us start by assuming that~$u_0$ is frequency localized in the unit ball. Then by the restriction inequality~\eqref{restSigduale}, and \eqref{lpislp} we have for any~$2\leq p \leq q \leq \infty$ 
 $$ 
 \|u\|_{L^\infty_s L_t^{q} L^{p}_{Y}} \leq C \| \mathcal F_{\H} u_0 \|_{L^{2}(\wh \H^d)} = C \|  u_0 \|_{L^{2}(\H^d)} \, ,
 $$
 where we used Plancherel formula.
 Now if~$u_0$ is frequency localized in the ball~${\mathcal B}_\Lambda$, then by virtue of\refeq{newFourierconvoleq-1} the function
 $$
 u_{0,\Lambda}\eqdefa u_0 \circ \d_{\Lambda^{-1}} $$
is frequency localized in~${\mathcal B}_1$,  and  it gives rise to the solution
$$
 u_{\Lambda}(t,Y,s)\eqdefa u ( \Lambda^{-2}t,\Lambda^{-1}Y, \Lambda^{-2}s)
$$
of $(S_\H)$. Since
$$
\|u_{\Lambda}\|_{L^\infty_s L_t^{q} L^{p}_{Y}} =\Lambda^{\frac2q + \frac{2d}p} \|u\|_{L^\infty_s L_t^{q} L^{p}_{Y}}
$$
and
$$
\|u_{0,\Lambda}\|_{L^{2}(\H^d)} =\Lambda^{d+1} \|u_0\|_{L^{2}(\H^d)}\, ,
$$
we infer that for~$u_0$   frequency localized in the ball~${\mathcal B}_\Lambda$, there holds
 $$ 
 \|u\|_{L^\infty_s L_t^{q} L^{p}_{Y}} \leq C\Lambda^{d+1-\frac2q-\frac{2d}p} \| u_0 \|_{L^{2}(\H^d)} \, .
 $$
The result follows from the fact that for any~$s \geq 0$,  if~$u_0$ is frequency localized in the ball~${\mathcal B}_\Lambda$, then
$$
\Lambda^s \|u_0 \|_{L^{2}(\H^d)} \lesssim \|u_0\|_{H^s(\H^d)} \, , 
$$
and we conclude the proof of the estimate by density of spectrally localized functions in~$H^s(\H^d)$.  This ends the proof of Theorem~\ref {STth}.\qed

 \subsection{Strichartz estimates for the wave operator on $\H^{d}$ }\label {stSchrwave}
 The aim of this section is to prove the Strichartz estimate stated in Theorem~\ref{STthwave}.  The method is identical to the previous section: again we reduce to the free wave equation and consider for~$(u_0,u_1)$    in~$\cS_{\rm {rad}}(\H^d)$ frequency localized in a unit ring in the sense of Definition\refer{freqlocgen},
   the Cauchy problem
$$
(W_\H)\qquad \left\{
\begin{array}{c}
 \partial_t^2 u -\D_\H u = 0\\
(u,\partial_t u)_{|t=0} = (u_0,u_1) \, .
\end{array}
\right. 
$$
Taking the Fourier transform we find that for all~$(t, \wh w)$ in~$\R \times \wt \H^d$,
 $$
\left\{
\begin{array}{rcl}
\ds   \frac{ d^2}   {dt^2}  \cF_\H(u) (t, n,m, \lam)&\! =\!  &- 4|\lam|(2|m|+d) \cF_\H(u) (t, n,m, \lam)\\
( \cF_\H u, \cF_\H\partial_t u)_{|t=0} &\!=\! &   {( \cF_\H u_0, \cF_\H u_1)}\, .
\end{array}
\right.
$$
By integration, this  leads to
 $$
\cF_\H(u) (t, n,m, \lam)=  \sum_{\pm} e^{\pm 2i t \sqrt{|\lam|(2|m|+d})}\cFH (\gamma_{\pm}) (|n|, |n|, \lam) \delta_{n,m}\, ,
$$
where similarly to~(\ref{defgammapm}) we have defined
$$
\cFH (\gamma_{\pm}) (n, n, \lam) \eqdefa\frac12 \Big(
\cFH  (u_0) \pm \frac1 {i\sqrt{4|\lambda| (2n+d)}}\cFH  (u_1) 
\Big)\, . 
$$
Then applying the  inverse Fourier  formula \eqref{inverseFourierH}, we infer that  the solution of the Cauchy problem $(W_\H)$ can be expressed as follows:
\beq
\label {solSH}u(t, Y, s)=  \sum_{\pm} \frac {2^{d-1}}  {\pi^{d+1} }   \int_{\wh \H^d} 
e^{is\lam} \, \cW(\wh w, Y) \, e^{\pm 2i t \sqrt{|\lam|(2|m|+d})} \, \cFH (\gamma_{\pm}) (|n|, |n|, \lam) \delta_{n,m} \, d\wh w  \, . \eeq
This can be written as the inverse Fourier transform in $\wh {\mathbb D}$ of $\cFH (\gamma_{\pm}) \, d\Sigma_{\pm}$, where  we recall that
by Remark~\ref{remarkwaves},
 for any regular function $\Phi $ on $\wh {\mathbb D}$
\beq
\label {formalmeasure} \int_{\wh {\mathbb D}} \Phi(\alpha,\wh w) \, d\Sigma_{\pm}(\alpha,\wh w)= \int_{\wh\H^d} \Phi(\sqrt{4|\lam|(2|m|+d}),\wh w) \,d\wh w \, .\eeq

\smallskip By hypothesis~$\gamma_{\pm}$ is frequency localized in a unit ring. Then  for any~$2\leq p \leq q \leq \infty$,  we have by virtue of the restriction inequality~\eqref{restSigduale} $$ 
 \|u\|_{L^\infty_s L_t^{q} L^{p}_{Y}} \leq C \| {\mathcal F}_\H^{-1}\gamma_{\pm} \|_{L^{2}(\H^d)} \, .
 $$
 Next if~$\gamma_{\pm}$ is frequency localized in the ring~$\mathcal{C}_\Lambda$, then by  \refeq{newFourierconvoleq-1} the function
 $$
 u_{0,\Lambda}\eqdefa u_0 \circ \d_{\Lambda^{-1}} $$
is frequency localized in the ring~$\mathcal{C}_1$ and gives rise to the solution
$$
 u_{\Lambda}(t,Y,s)\eqdefa u ( \Lambda^{-1}t,\Lambda^{-1}Y, \Lambda^{-2}s)
$$
of $(W_\H)$. Since
$$
\|u_{\Lambda}\|_{L^\infty_s L_t^{q} L^{p}_{Y}} =\Lambda^{\frac1q + \frac{2d}p} \|u\|_{L^\infty_s L_t^{q} L^{p}_{Y}}
$$
and
$$
\|u_{0,\Lambda}\|_{L^{2}(\H^d)} =\Lambda^{d+1} \|\gamma_{\pm}\|_{L^{2}(\H^d)}\, ,
$$
we infer that for~$u_0$   frequency localized in the ring~$\cC_\Lambda$, there holds
 $$ 
 \|u\|_{L^\infty_s L_t^{q} L^{p}_{Y} } \leq C\Lambda^{d+1-\frac1q-\frac{2d}p} \| \gamma_{\pm} \|_{L^{2}(\H^d)} \, .
 $$
 Then by Bernstein's Lemma~\ref{bernstein}, and in particular estimate~(\ref{eq:lech3}), followed by   Plancherel's inequality we infer that
 $$
  \|u\|_{L^\infty_s L_t^{q} L^{p}_{Y}} \leq  C\Lambda^{d-\frac1q-\frac{2d}p}  
  \Big( \| \nabla_{\H^d}u_0 \|_{L^{2}(\H^d)}  +\| u_1 \|_{L^{2}(\H^d)} \Big)\, .
 $$
To conclude the proof of the estimate, we argue as for the Schr\"odinger equation making use of  
the fact that for any real~$s$,  if a function~$f$ is frequency localized in the   ring~${\mathcal C}_\Lambda$, then
$$
\Lambda^s \|f \|_{L^{2}(\H^d)} \leq C \|f\|_{H^s(\H^d)} \, .
$$
 This proves Theorem~\ref{STthwave}.
 \qed

\appendix

\section{Proof of a technical result}\label{proofprop}

We have used the following lemma due to M\"uller (see\ccite{Muller}), whose proof we sketch below.
  \begin{lemma}
\label {LemmeMuller}
{\sl Defining for  $f$, $g$ in  $\cS(T^\star\R^d)$ and $\lam$ in $\R\setminus\{0\}$ the $\lam$-twisted convolution
\beq
\label  {twisted}
(f \star_\lam g )(Y)\eqdefa \int_{T^\star\R^d} f(Y-w) g(w) e^{2i\lam \sigma(Y,w)}dw\, , 
\eeq
with~$\sigma$ defined in~{\rm(\ref{defsigma})},  the following estimate  holds: there exists a positive constant $C_{d}$ such that 
\beq
\label {est2} \|f \star_\lam \wt \cW(\ell,\lam,\cdot) \|_{L^{p}(T^\star\R^d)}\leq C_{d} \, |\lam|^{- \frac {2d} {p'}}\, \ell^ {(d-1)(1-\frac {2} {p'})}\, \|f \|_{L^{p}(T^\star\R^d)}\, , 
\eeq 
for all $1\leq p \leq 2$ and all integers $\ell \geq 1$,  where the function $ \wt \cW(\ell,\lam,Y)$ is defined by~\eqref{princident}. }
\end{lemma}

\begin{proof}[Proof of Lemma {\rm\ref{LemmeMuller}}]\label{prooflemma}
Let us start by  establishing that    for $f$, $g$ in $\cS(T^\star\R^d)$ and  $1\leq p \leq 2$, the following estimate holds:
 $$
 \|f \star_\lam g \|_{L^{p'}(T^\star\R^d)}\leq C_{p,d} \, |\lam|^{- \frac {d} {p'}} \, \|f \|_{L^{p}(T^\star\R^d)} \|g \|^{\frac {2} {p'}} _{L^{2}(T^\star\R^d)} \|g \|^{1-\frac {2} {p'}} _{L^{\infty}(T^\star\R^d)}\, . $$  
By definition,  
$$
(f \star_\lam g )(Y)= \int_{T^\star\R^d} f(Y-w) g(w) e^{2i\lam \sigma(Y,w)}dw\, , 
$$
 which easily implies by  Young's inequalities that 
 $$
 \|f \star_\lam g \|_{L^{\infty}(T^\star\R^d)}\leq   \|f \|_{L^{1}(T^\star\R^d)} \|g \|_{L^{\infty}(T^\star\R^d)}\, . 
 $$  
 Therefore invoking the method of real   interpolation, we are reduced  to showing that 
 $$
 \|f \star_\lam g \|_{L^{2}(T^\star\R^d)}\leq C_{d} \, |\lam|^{- \frac {d} {2}} \, \|f \|_{L^{2}(T^\star\R^d)} \|g \|_{L^{2}(T^\star\R^d)}\, ,
 $$ 
and this follows by an easy dilation argument from the well-known fact that $(L^2(T^\star\R^d), \star_1)$ is a Hilbert algebra (see for instance\ccite{howe}). 
 
 \medskip Let us now focus on Estimate \eqref{est2}.   We first recall that 
   $$
  \wt \cW(\ell,\lam,Y) = \sum_{|n|= \ell} \cL_n(|\lam|^{\frac {1} {2}} \, Y)\, , 
  $$ where for $Z=(Z_1, \cdots, Z_d)$ in $T^\star\R^d$, we denote $\cL_n(Z)=  e^{-|Z|^2} \Pi^d_{j=1} L_ {n_j}( 2 |Z_j|)$ with $L_ {n_j}$ the  Laguerre polynomial of order $n_j$ and type $0$.

  Define now for $n \in \N^d$ the operator $T_n$   on $L^{2}(T^\star\R^d)$ by
\begin{equation}
\label{defTn} T_n f\eqdefa f  \star_\lam \cL_n(|\lam|^{\frac {1} {2}} \cdot)\, \virgp \end{equation}
 so that 
\begin{equation}
\label{defT}
 Tf= f \star_\lam \wt \cW(\ell,\lam,\cdot)= \sum_{|n|= \ell} T_n f \, .
\end{equation}
 Then using the fact that the Laguerre polynomials $(L_k)_{k \in \N}$ are pairwise orthogonal on~$[0, \infty[$ with respect to the measure $e^{-x} dx$, we infer that the family of  operators $(T_n)_{|n|= \ell}$ is also pairwise orthogonal, and thus denoting by $ \| T \|$ the norm of $T$ defined by \eqref{defT} as an operator on $L^{2}(T^\star\R^d)$, we can conclude that 
\begin{equation}
\label{normT} \| T \|_{\cL(L^2(T^\star\R^d))} = \max_{|n|= \ell}\| T_n \|_{\cL(L^2(T^\star\R^d))} \, .\end{equation}
 Since $\| \cL_n(|\lam|^{\frac {1} {2}} \cdot)\|_{L^2(T^\star\R^d)} = |\lam|^{- \frac {d} {2}} \, \| \cL_n\|_{L^2(T^\star\R^d)}$, 
  one obtains 
$$\| T_n \|_{\cL(L^2(T^\star\R^d))} \lesssim    |\lam|^{- d} \,  .$$ 
 In view of \eqref{normT}, this implies that 
  $$\|f \star_\lam \wt \cW(\ell,\lam,\cdot) \|_{L^{2}(T^\star\R^d)}\lesssim  |\lam|^{- d} \,  \|f \|_{L^{2}(T^\star\R^d)}\, .  $$
Finally, by definition  
$$\|\wt \cW(\ell,\lam,\cdot) \|_{L^{\infty }(T^\star\R^d)}\lesssim \ell^{d-1} \,  ,$$
which ensures that 
 $$\|f \star_\lam \wt \cW(\ell,\lam,\cdot) \|_{L^{\infty}(T^\star\R^d)}\lesssim  \ell^{d-1} \,   \|f \|_{L^{1}(T^\star\R^d)}\, ,  $$
 and completes the proof of Estimate \eqref{est2} by interpolation.

 \end{proof}
 
 \section{The inhomogeneous case}\label{proofprop}
 Denoting by~$(\mathcal U(t))_{t\in \R} $ the solution operator of the Schr\"odinger equation on the Heisenberg group, namely~$\mathcal U(t)u_0$ is the solution of~$(S_\H)$ with $f=0$ at time $t$ associated with the data~$u_0$, then similarly to the euclidean case   $(\mathcal U(t))_{t\in \R} $  is a one-parameter group of unitary operators on $L^2(\H^d)$. Moreover,   the solution to the inhomogeneous  equation  $$
  \left\{
\begin{array}{c}
i\partial_t u -\D_\H u = f\\
u_{|t=0} = 0\,,
\end{array}
\right. $$
 writes 
\begin{equation}
\label{estlinh} u(t, \cdot)= -i \int^t_0 \mathcal U(t-t') f(t', \cdot) dt'  ,\end{equation}
Let us check that it satisfies, for all  admissible pairs $(p,q)$  in $
\cA^{\mbox{\tiny{S}}}$,  
\begin{equation}
\label{estlinhst} \|u\|_{L^\infty_s L_t^{q} L^{p}_{Y}} \lesssim \|f \|_{L_t^1 H^{\sigma}(\H^d)} \end{equation}
with $\ds \sigma= \frac Q2 - \frac2q-\frac{2d}p\cdotp $ 

 Let us first assume that, for all $t$,  the source term $f(t, \cdot) $ is frequency localized in  in  the unit ball~${\mathcal B}_1$ in the sense of Definition\refer{freqlocgen}, and  recall that according to the results of Section\refer{stSchr}, if $g$ is frequency localized in a unit ball, then  for all~$2\leq p \leq q \leq \infty$ 
\begin{equation}
\label{estl2}
 \|\mathcal U(t)g\|_{L^\infty_s L_t^{q} L^{p}_{Y}} \lesssim \|  g \|_{L^{2}(\H^d)} \, .
\end{equation}
Taking advantage of\refeq{estlinh}, we have 
 for all $s \in \R$, $$\|u(t, \cdot, s) \|_{L^{p}_{Y}} \leq \int_\R \|\mathcal  U(t) \mathcal U(-t') f(t', \cdot, s) \|_{L^{p}_{Y}} dt' . $$
 Therefore, still  for all $s$,
 $$\|u(\cdot,\cdot, s) \|_{L_t^{q} L^{p}_{Y}} \leq \int_\R   \|\mathcal  U(\cdot) \mathcal U(-t') f(t', \cdot, s) \|_{L_t^{q} L^{p}_{Y}} dt' .$$ 
 Invoking\refeq{estl2}, we deduce that 
  $$\|u\|_{L^\infty_s L_t^{q} L^{p}_{Y}} \leq \int_\R \|\mathcal U(-t') f(t', \cdot) \|_{{L^{2}(\H^d)}} dt' .$$ 
Since   $\mathcal U(-t')$ is unitary on $L^{2}(\H^d)$, we readily gather that 
\begin{equation}
\label{estlinhgamma} \|u \|_{L^\infty_s L_t^{q} L^{p}_{Y}} \leq \int_\R \|f(t', \cdot) \|_{{L^{2}(\H^d)}} dt' .\end{equation}
 Now if for all $t$,~$f(t, \cdot)$ is frequency localized in a ball of size $\Lambda$, then setting
  $$
 f_{\Lambda}(t, \cdot) \eqdefa \Lambda^{-2} f (\Lambda^{-2} t, \cdot) \circ \d_{\Lambda^{-1}} $$
we find that on the one hand, $f_{\Lambda}(t, \cdot)$ is frequency localized in a unit ball for all $t$,  and  on the other hand  that the solution to the Cauchy problem $$
  \left\{
\begin{array}{c}
i\partial_t u_{\Lambda} -\D_\H u_{\Lambda} = f_{\Lambda}\\
u_{|t=0} = 0\,,
\end{array}
\right. $$
writes
$$u_{\Lambda}(t,w)= u (\Lambda^{-2} t, \cdot) \circ \d_{\Lambda^{-1}} \,.$$
Now by scale invariance, we have
$$\int_\R \|f_{\Lambda}(t', \cdot) \|_{{L^{2}(\H^d)}} dt' = \Lambda^{\frac{Q}2}\int_\R \|f(t', \cdot) \|_{{L^{2}(\H^d)}} dt' $$
and 
$$
\|u_{\Lambda}\|_{L^\infty_s L_t^{q} L^{p}_{Y}}  =\Lambda^{\frac{2} {q} + \frac{2d} {p}}  \|u\|_{L^\infty_s L_t^{q} L^{p}_{Y}} \, .
$$
Consequently, we get 
 $$ 
 \|u\|_{L^\infty_s L_t^{q} L^{p}_{Y}} \leq C  \int_\R \Lambda^{\frac{Q}2- \frac{2} {q} - \frac{2d} {p}} \|f(t', \cdot) \|_{{L^{2}(\H^d)}} dt' \, .
 $$ 
Since  $\ds \frac{Q}2- \frac{2} {q} - \frac{2d} {p} \geq 0$,  we have 
$$\Lambda^{\frac{Q}2- \frac{2} {q} - \frac{2d} {p}} \|f(t', \cdot) \|_{{L^{2}(\H^d)}} \lesssim \|f(t', \cdot) \|_{{H^{\frac{Q}2- \frac{2} {q} - \frac{2d} {p}}(\H^d)}} \,, $$ which  completes the proof of\refeq{estlinhst}.

  \bigskip 
  In the case of the wave equation, we have   seen in Paragraph\refer{stSchrwave} that    when the Cauchy data $u_0$ and $u_1$ are frequency localized in a ring, then $u$ the solution to the Cauchy problem~$(W_\H)$ with $f=0$ reads, with the previous notations, 
  $$ u(t, \cdot)= \mathcal U^+ (t) \gamma^+ + \mathcal U^- (t) \gamma^- \, .$$
  This allows to investigate  the inhomogeneous wave equation  by similar arguments than  the Schr\"odinger equation dealt with above. The proof is left to the reader.

\end{document}